\documentclass[11pt,reqno]{amsart}

\usepackage{hyperref}
\usepackage{graphicx}
\usepackage{color}
\usepackage{amsfonts,amssymb}
\usepackage{amsthm}
\usepackage{psfrag}                     

\usepackage{mathrsfs}

\usepackage{a4wide}

\usepackage[alphabetic,initials]{amsrefs}

\newtheorem{neu}{}[section]
\newtheorem*{Cor*}{Corollary}
\newtheorem{Thm}{Theorem}
\newtheorem*{Thm*}{Theorem}
\newtheorem{Cor}[Thm]{Corollary}
\newtheorem{Prop}[neu]{Proposition}
\newtheorem*{Prop*}{Proposition}
\theoremstyle{definition}
\newtheorem{Lemma}[neu]{Lemma}
\newtheorem*{Rmk*}{Remark}
\newtheorem*{Rmks*}{Remarks} 
\newtheorem{Rmk}[neu]{Remark}
\newtheorem{Ex}[neu]{Example}
\newtheorem*{Ex*}{Example}

\newtheorem{Assumption}[neu]{Assumption}
\newtheorem{Def}[neu]{Definition}

\newtheorem{Conj}{Conjecture} 
\newtheorem{Question}{Question}

\newcommand{\N}{\mathbb{N}}

\newcommand{\Q}{\mathbb{Q}}
\newcommand{\Z}{\mathbb{Z}}
\newcommand{\R}{\mathbb{R}}
\newcommand{\C}{\mathbb{C}}

\newcommand{\CZ}{\mu_{\mathrm{CZ}}}

\newcommand{\im}{\mathrm{im\,}}

\newcommand{\ind}{\mathrm{ind\,}}

\newcommand{\om}{\omega}

\newcommand{\slope}{\operatorname{slope}}                               
\newcommand{\Sigmadot}{\dot{\Sigma}}                                    
\newcommand{\sign}{\operatorname{sign}}                                 
\newcommand{\T}{\mathcal{T}}                                            
\newcommand{\wind}{\operatorname{wind}}                                 
\newcommand{\x}{\textbf{x}}                                             

\newcommand{\U}{\mathcal{U}}

\newcommand{\D}{\mathbb{D}}

\newcommand{\J}{\mathcal{J}}

\newcommand{\V}{\mathcal{V}}
\newcommand{\W}{\mathcal{W}}

\newcommand{\beq}{\begin{equation}}
\newcommand{\beqn}{\begin{equation}\nonumber}
\newcommand{\eeq}{\end{equation}}

\newcommand{\bea}{\begin{equation}\begin{aligned}}
\newcommand{\bean}{\begin{equation}\begin{aligned}\nonumber}
\newcommand{\eea}{\end{aligned}\end{equation}}

\numberwithin{equation}{section}

\definecolor{blue}{rgb}{0,0,1}
\definecolor{red}{rgb}{1,0,0}
\definecolor{green}{rgb}{0,.7,0}

\newcommand{\cc}{{\mathbf{c}}}
\renewcommand{\dbar}{\bar{\partial}}
\newcommand{\dist}{\operatorname{dist}}
\newcommand{\p}{\partial}
\newcommand{\nN}{\mathcal{N}}
\newcommand{\windpi}{\operatorname{wind}_\pi}

\newcommand{\mM}{{\mathcal M}}
\newcommand{\reg}{{\operatorname{reg}}}
\newcommand{\tT}{{\mathcal T}}
\newcommand{\Contact}{\Xi(3)}
\newcommand{\Embed}{\Xi_{\text{\textsf{embed}}}(3)}
\newcommand{\Nonsep}{\Xi_{\text{\textsf{nonsep}}}(3)}

\begin{document}
\title{On Non-Separating Contact Hypersurfaces in Symplectic $4$--Manifolds}
\author{Peter Albers}
\author{Barney Bramham}
\author{Chris Wendl}

\address{
    Peter Albers\\
    Department of Mathematics\\
    ETH Z\"urich}
\email{palbers@math.ethz.ch}
\address{
    Barney Bramham\\
    Max-Planck Institute, Leipzig}
\email{bramham@mis.mpg.de}
\address{
    Chris Wendl\\
    Department of Mathematics\\
    ETH Z\"urich}
\email{wendl@math.ethz.ch}
\begin{abstract}
We show that certain classes of contact $3$--manifolds do not admit 
non-separating contact type embeddings into any closed symplectic
$4$--manifolds, e.g.~this is the case for all contact manifolds that
are (partially) planar or have Giroux torsion.  The latter implies
that manifolds with Giroux torsion do not admit contact type embeddings into any
closed symplectic $4$--manifolds.  Similarly, there are symplectic
$4$--manifolds that can admit smoothly embedded non-separating hypersurfaces,
but not of contact type: we observe that this is the case for all
symplectic ruled surfaces.
\end{abstract}
\maketitle

\section{Introduction}\label{sec:introduction}

\subsection{Main results}\label{subsec:main}

Let $(W,\omega)$ denote a closed symplectic
manifold of dimension four.  A closed hypersurface
$M \subset W$ is of contact type if {it is transverse to a Liouville vector
field, i.e.~a smooth vector field $Y$ defined near $M$ such that
$L_Y\omega = \omega$.}  Then $\iota_Y\omega$
is a contact form on $M$, and we will denote the resulting contact structure by
$\xi = \ker\iota_Y\omega$; it is independent of $Y$ up to isotopy.  
If $M$ separates $W$ into two components, 
then it is said to form a {convex} boundary on the component where $Y$ 
points outward, and a {concave} boundary on the other component.
By constructions due to Etnyre-Honda \cite{EtnyreHonda:cobordisms}
and Eliashberg \cite{Eliashberg:cap},
every contact $3$--manifold can occur as the concave boundary of some
compact symplectic manifold.  This is not true for convex boundaries:
for instance, Gromov \cite{Gromov} and Eliashberg 
\cite{Eliashberg:diskFilling}
showed that overtwisted contact manifolds can never occur as convex
boundaries, and a finer obstruction comes from {Giroux torsion}
\cite{Gay:GirouxTorsion}.

In this paper, we address the question of whether a given contact
$3$--manifold $(M,\xi)$ can occur as a {\emph{non-separating}}
contact hypersurface in any closed symplectic manifold, and similarly,
whether a given symplectic $4$--manifold $(W,\omega)$ admits non-separating
contact hypersurfaces.  Observe that \emph{separating} contact hypersurfaces
always exist in abundance, e.g.~the boundaries of balls in Darboux
neighborhoods.  We will see in Example~\ref{ex:EtnyresExample} that 
non-separating contact hypersurfaces sometimes exist, but there are
restrictions, as the following Theorem shows.  


\begin{Thm}\label{thm:consequence}
Suppose $(M,\xi)$ is a closed contact $3$--manifold which has any one of the following properties:
\begin{enumerate}
\item $(M,\xi)$ has Giroux torsion
\item $(M,\xi)$ is planar or partially planar (see Definition~\ref{def:partiallyPlanar} below)
\item $(M,\xi)$ admits a symplectic cap containing a symplectically embedded
sphere of nonnegative self-intersection number
\end{enumerate}
Then every contact type embedding of $(M,\xi)$ into any closed symplectic
$4$--manifold is separating.
\end{Thm}

\begin{Rmk}\label{rmk:generalizing_main_thm}
Theorem \ref{thm:consequence} admits an easy generalization as follows.
We will say that $(M,\xi)$ has any given property
\emph{after contact surgery} if the property holds for some contact manifold
$(M',\xi')$ obtained from $(M,\xi)$ by a (possibly trivial) sequence of contact
connected sum operations and contact $(-1)$--surgeries.  The significance
of these operations (see e.g.~\cite{Geiges:book})
is that they imply the existence of a symplectic cobordism
from $(M,\xi)$ to $(M',\xi')$: recall 
that a symplectic cobordism from $(M_-,\xi_-)$ to $(M_+,\xi_+)$ is in general a
compact symplectic manifold $(W,\omega)$ with {$\p W = (-M_-) \sqcup M_+$},
such that there is a Liouville vector field near $\p W$ defining
$(M_-,\xi_-)$ and $(M_+,\xi_+)$ as concave and convex boundary components
respectively.  The special case where $M_- = \emptyset$ is a convex
filling of $(M_+,\xi_+)$. {If} $M_+ = \emptyset$ we instead get a concave
filling of $(M_-,\xi_-)$, also known as a symplectic cap.

It will follow from the more general Theorem~\ref{thm:separating} below
that Theorem~\ref{thm:consequence} also holds whenever properties~(1)
or~(2) hold after contact surgery.  (For property~(3) this statement is
trivial.)
\end{Rmk}

The following example shows that non-separating 
contact type hypersurfaces do exist in general.

\begin{Ex}[\textbf{Etnyre}]
\label{ex:EtnyresExample}
Suppose $(W_0,\omega_0)$ is a compact symplectic manifold with a convex
boundary that has two connected components.  {In this case we say that
$(W_0,\omega_0)$ is a convex semifilling of each of its boundary components;
the} existence of such
objects was first established by McDuff \cite{McDuff:boundaries}.
Produce a new symplectic manifold $(W_1,\omega_1)$ with convex boundary
by attaching a symplectic $1$--handle along a pair of $3$--balls in
different components of~$\p W_0$.  Now cap $W_1$ with a concave
filling of $\p W_1$ as provided by \cite{EtnyreHonda:cobordisms}:
this produces a closed symplectic
manifold $(W,\omega)$, which contains both of the components of
$\p W_0$ as non-separating contact hypersurfaces
(see Figure~\ref{fig:Etnyre}).
\end{Ex}

\begin{figure}[hbt]
 \centering
 \includegraphics[scale=0.25, bb=0 0 998 300]{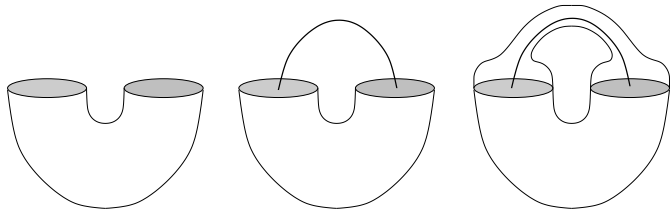}
 \caption{\label{fig:Etnyre} The construction from 
 Example~\ref{ex:EtnyresExample} of a symplectic manifold with non-separating
 contact hypersurfaces.}
\end{figure}

The example demonstrates that $(M,\xi)$ can occur as a 
non-separating hypersurface
in some closed symplectic manifold
whenever it arises from a convex filling with disconnected boundary.
There are, however, contact manifolds that never arise in this way:
McDuff \cite{McDuff:boundaries} showed that this is the case for the
tight $3$--sphere, and the result was generalized by Etnyre 
\cite{Etnyre:planar} to all planar contact manifolds, i.e.~{those which}
are supported by planar open books.  The latter suggests that planar open books
may provide an obstruction to non-separating contact embeddings,
and this is indeed true due to Theorem~\ref{thm:consequence}.
As we'll see shortly, there are also non-planar contact manifolds
(e.g.~the standard contact $3$--torus) which
satisfy the assumptions of Theorem~\ref{thm:consequence}, and thus also
the following corollary:

\begin{Cor}
\label{cor:semifillings}
Given the assumptions of Theorem~\ref{thm:consequence} (see also Remark \ref{rmk:generalizing_main_thm}), every
convex semifilling of $(M,\xi)$ has connected boundary.
\end{Cor}
Actually one can use the same methods to give a slightly simpler proof of 
Corollary~\ref{cor:semifillings} which is independent of the theorem;
we'll do this in \S\ref{sec:proofs}.  

In the case of Giroux torsion, a result of Gay 
\cite{Gay:GirouxTorsion} shows that $(M,\xi)$ does not admit any convex 
fillings,\footnote{An alternative proof closely related to the arguments in
this paper appears in \cite{Wendl:fillable}.} thus Theorem~\ref{thm:consequence}
has the following stronger consequence:

\begin{Cor}
\label{cor:torsion}
If $(M,\xi)$ has Giroux torsion (possibly after contact surgery),
then it does not admit a contact embedding
into any closed symplectic $4$--manifold.
\end{Cor}

Theorem~\ref{thm:consequence} will follow from some more
technical results stated in \S\ref{subsec:star}, which also includes a more
general statement involving contact hypersurfaces in a symplectic manifold
with convex boundary.  The unifying idea can be summarized as follows.
Whenever a non-separating hypersurface $M \subset W$
exists, one can use it to construct a special noncompact symplectic manifold
$(\V,\om)$ with convex boundary~$M$.  We do this by first cutting $W$ open
along $M$ to produce a symplectic cobordism $(V_1,\om)$ {from a concave copy of $M$ to a convex copy of $M$}, and then removing 
the concave boundary by attaching an infinite chain of copies of
$(V_1,\om)$ along matching concave and convex boundaries; a picture of this
construction appears as Figure~\ref{fig:infinitechain} in \S\ref{sec:proofs},
where it is explained in detail.
Now our assumptions on $(W,\om)$ or $(M,\xi)$
guarantee the existence of an embedded holomorphic curve in $(\V,\om)$ with
certain properties: in particular, we'll show in \S\ref{sec:compactness}
that this curve belongs to a smooth and compact $2$--dimensional moduli
space of curves that foliate $(\V,\om)$.  But this would imply that
$(\V,\om)$ is compact, and thus yields a contradiction.

\begin{Rmk}
A contact manifold $(M,\xi)$ is said to be {weakly fillable} if it occurs as
the boundary of a compact symplectic manifold $(W,\omega)$ such that
$\omega|_{\xi} > 0$ on $\p W$.  A fundamental result of Eliashberg
\cite{Eliashberg:diskFilling} and Gromov \cite{Gromov} shows that
{overtwisted} contact manifolds are never weakly fillable: the original
proof is based on the existence of a so-called {Bishop family} of 
pseudoholomorphic disks with
boundary on an overtwisted disk in $\p W$, and derives a contradiction using
Gromov compactness (a complete exposition may be found in
\cite{Zehmisch:diplom}).  In the setting described above,
one can adapt the Eliashberg-Gromov
argument to show that overtwisted contact manifolds do not occur as
hypersurfaces of {weak} contact type in any closed symplectic
manifold.  If we remove the word ``weak'', then this is also implied by
Corollary~\ref{cor:torsion} since overtwisted contact manifolds have
infinite Giroux torsion.
\end{Rmk}

The third condition in Theorem~\ref{thm:consequence} is
satisfied by any contact $3$--manifold that has a contact
embedding into the standard symplectic $\R^4$: indeed, the latter can be 
identified with 
$\C P^2 \setminus \C P^1$, and $\C P^1$ is a symplectically embedded
sphere with self-intersection~$1$.  As Yasha Eliashberg has pointed out to us,
Theorem~\ref{thm:consequence} in this case also morally follows, via the 
infinite chain construction sketched above, from Gromov's classification
\cite{Gromov} of symplectic manifolds that are Euclidean at infinity---one 
just has to be a little more careful in the noncompact setting 
(cf.~Prop.~\ref{prop:uniform_compactness}).
Natural examples are the unit cotangent
bundles of all closed surfaces that admit Lagrangian embeddings into
$\R^4$, i.e.~the torus, and the connected sums of the Klein bottle with 
{a positive number of} oriented surfaces of positive, even genus.  
Further examples of symplectic caps containing nonnegative symplectic spheres
have appeared in the work of Ohta-Ono et al \cites{OhtaOno:simpleSingularities,
BhupalOno} on contact manifolds obtained from algebraic surface singularities.

We now explain the notion of a partially planar contact manifold,
which is due to the third author (see \cite{Wendl:fiberSums}).
Recall that an open book decomposition for $M$ consists of the data $(B,\pi)$ 
where $B\subset M$ is an oriented link, and $\pi:M \setminus B \to S^1$
is a fibration for which each fiber $\pi^{-1}(\mbox{point})$ is an 
embedded surface whose closure in $M$ has oriented boundary~$B$.  These fibers
are called the {pages} of the open book $(B,\pi)$, and 
$B$ is called the {binding}.  We recall the following important concept  
introduced by Giroux \cite{Giroux:openbook}. 

\begin{Def}\label{def:supporting open book}
A contact structure $\xi$ on $M$ is said to be supported by an open 
book decomposition $(B,\pi)$ if it admits a contact form $\lambda$ 
such that the associated Reeb vector field is positively transverse to 
the pages and is positively tangent to the link~$B$.  
\end{Def}

In particular, the component circles of $B$ are closed Reeb orbits for such a 
contact form $\lambda$.  These are referred to as the {binding orbits}.  

\begin{Def}\label{def:planar}
A contact manifold $(M,\xi)$ is said to be {planar} if it admits a supporting
open book decomposition for which each page has genus zero.   
\end{Def}

Giroux established that every contact structure on a closed 3-manifold 
is supported by some open book decomposition.  Entyre showed in 
\cite{Etnyre:planar} that all overtwisted contact structures are planar,
though not all contact structures are.

The notion of a planar contact manifold can be generalized using
the {contact fiber sum}; the following is a special case of a construction
originally due to Gromov \cite{Gromov:PDRs} and Geiges \cite{Geiges:constructions}
(see also \cite{Geiges:book}).  For $i=1,2$, suppose 
$(M_i,\xi_i)$ are contact manifolds with supporting open book decompositions
$\pi_i : M_i \setminus B_i \to S^1$, and $\gamma_i \subset B_i$ are
connected components of the bindings.  Each $\gamma_i$ is a transverse knot,
thus one can identify neighborhoods $\nN(\gamma_i)$ with solid tori
via an orientation preserving map
$$
\Phi : \nN(\gamma_1) \cup \nN(\gamma_2) \to S^1 \times \D,
$$
thus defining coordinates $(\theta,\rho,\phi)$, where $\theta \in S^1$
and $(\rho,\phi)$ are polar coordinates on~$\D$ (for simplicity we shall
take $\phi \in S^1 = \R / \Z$, thus the actual angle is this times $2\pi$).
We will assume
without loss of generality (and perhaps after a small isotopy of the open
books) that these coordinates have the following properties:
\begin{enumerate}
\item
The contact structure $\xi_i$ is the kernel of
$\lambda_i = f(\rho)\ d\theta + g(\rho)\ d\phi$ for some pair of
functions $f$ and $g$ with $f(0) > 0$ and $g(0) = 0$.
\item
The pages of $\pi_i$ have the form $\{ \phi = \text{const} \}$ near~$\gamma_i$.
\end{enumerate}
Note that the contact condition requires $f(\rho)g'(\rho) - f'(\rho)g(\rho)
> 0$ for $\rho > 0$ and $g''(0) > 0$.
Using these choices, a new contact manifold
$$
(M_1,\xi_1) \#_\Phi (M_2,\xi_2)
$$
can be defined in two steps:
\begin{enumerate}
\item[(i)]
Modify $(M_i,\xi_i)$ by ``blowing up'' $\gamma_i$ to produce a contact manifold
$(\widehat{M}_i,\hat{\xi}_i)$ with pre-Lagrangian torus boundary: we do this
by removing a solid torus neighborhood $\{ \rho \le \epsilon\}$ and replacing
it with $S^1 \times [0,\epsilon] \times S^1$ by the natural identification
of the coordinates $(\theta,\rho,\phi) \in S^1 \times [0,\epsilon] \times S^1$.
We also modify $\lambda_i$ for $\rho \in [0,\epsilon)$ to define a smooth
contact form near $\p\widehat{M}_i$ by making $C^0$--small changes to 
$f$ and $g$ so that they become restrictions of even and odd functions 
respectively, with $g'(0) > 0$.  
In terms of the Reeb vector field defined by $\lambda_i$,
the result of this change is to replace the single Reeb orbit originally at
$\{\rho = 0\}$ by a torus $S^1 \times S^1$ foliated by Reeb orbits of the
form $S^1 \times \{\text{pt}\}$.
\item[(ii)]
Attach $(\widehat{M}_1,\hat{\xi}_1)$ to $(\widehat{M}_2,\hat{\xi}_2)$ along their
boundaries as follows: first, define new coordinates $(\hat{\theta},
\hat{\rho},\hat{\phi}) \in S^1 \times \R \times S^1$ 
near $\p\widehat{M}_i$ so that they are the same as
the old coordinates on $\widehat{M}_1$, but on $\widehat{M}_2$ we set
$$
(\hat{\theta},\hat{\rho},\hat{\phi}) := (\theta,-\rho,-\phi),
$$
so $\hat{\rho} \le 0$ near $\p\widehat{M}_2$.  We now attach
$\widehat{M}_1$ to $\widehat{M}_2$ via {a diffeomorphism} such that
$(\hat{\theta},\hat{\rho},\hat{\phi}) \in S^1 \times [-\epsilon,\epsilon]
\times S^1$ become well defined coordinates after attaching.  Our
assumptions on the modified functions $f$ and $g$ imply also that
$f(\hat{\rho})\ d\hat{\theta} + g(\hat{\rho})\ d\hat{\phi}$ gives a
smooth contact form on $M_1 \#_\Phi M_2$ which matches the original
outside the region $\{ \hat{\rho} \in (-\epsilon,\epsilon) \}$.
\end{enumerate}

In a straightforward way, one can generalize this definition to a sum
of two or more open books on contact manifolds
$(M_1,\xi_1),\ldots,(M_N,\xi_N)$ along multiple binding components: then each 
of these components becomes
a boundary component in its respective ``blown up'' manifold $\widehat{M}_i$, and
it becomes a special pre-Lagrangian torus in the sum 
$$
\#_\Phi (M_i,\xi_i).
$$

\begin{Def}
\label{def:partiallyPlanar}
We say that $(M,\xi)$ is {partially planar} if it can be constructed in the above manner
as a contact fiber sum along binding orbits of open
book decompositions, at least one of which is planar.
\end{Def}

Obviously, every planar contact manifold is also partially planar.  Since
there exist contact $3$--manifolds that admit semifillings with disconnected
boundary, a consequence of Corollary~\ref{cor:semifillings} is now the 
following:
\begin{Cor}\label{cor:partiallyPlanar}
Not every contact manifold is partially planar.  
\end{Cor}

\begin{Ex}{\label{ex:notPartiallyPlanar}}
McDuff showed in \cite{McDuff:boundaries} that for any closed oriented surface
$\Sigma$ of genus at least two, if $ST^*\Sigma$ denotes the unit cotangent
bundle, then there is a symplectic structure on $[0,1] \times ST^*\Sigma$ which
is convex on the boundary and
induces the canonical contact structure at $\{1\} \times S T^*\Sigma$.
More generally, Geiges \cite{Geiges:disconnected} constructed a class of
closed $3$--manifolds $M$ which admit pairs of 
contact forms $\lambda_\pm$ such that 
$$
\lambda_+ \wedge d\lambda_+ = -\lambda_- \wedge d\lambda_- > 0
\quad\text{ and }\quad
\lambda_+ \wedge d\lambda_- = \lambda_- \wedge d\lambda_+ = 0.
$$
In this situation, $[0,1] \times M$ admits a symplectic structure such that
both boundary components are convex, giving a convex filling of
$(M,\ker\lambda_+) \sqcup (-M,\ker\lambda_-)$.  It follows from
Corollary~\ref{cor:semifillings} that none of these contact manifolds
are partially planar.  Moreover by Example~\ref{ex:EtnyresExample}, each
of them admits a non-separating contact type embedding into some closed
symplectic manifold.
\end{Ex}

The next example shows that there are also partially
planar contact manifolds that are not planar.

\begin{Ex}\label{ex:T3}
The standard contact $S^1 \times S^2$ is planar: it admits a supporting 
open book decomposition with two binding orbits connected by cylindrical
pages.  If we take two copies of this, pair up both of their respective
binding components and construct the fiber sum, we obtain the standard
contact $T^3$, which is not planar 
{due to a result of Etnyre \cite{Etnyre:planar}}.
In fact, each of the tight contact tori $(T^3,\xi_n)$, where
$\xi_n = \ker\left[ \cos(2\pi n \theta)\ dx + \sin(2\pi n\theta)\ dy \right]$
in coordinates $(x,y,\theta) \in S^1 \times S^1\times S^1$, can be obtained
as a fiber sum of $2n$ copies of the standard $S^1 \times S^2$; 
see Figure~\ref{fig:T3}.  By a result of Kanda \cite{Kanda:torus}, this
includes every tight contact structure on~$T^3$.
\end{Ex}

\begin{figure}[hbt]
\includegraphics{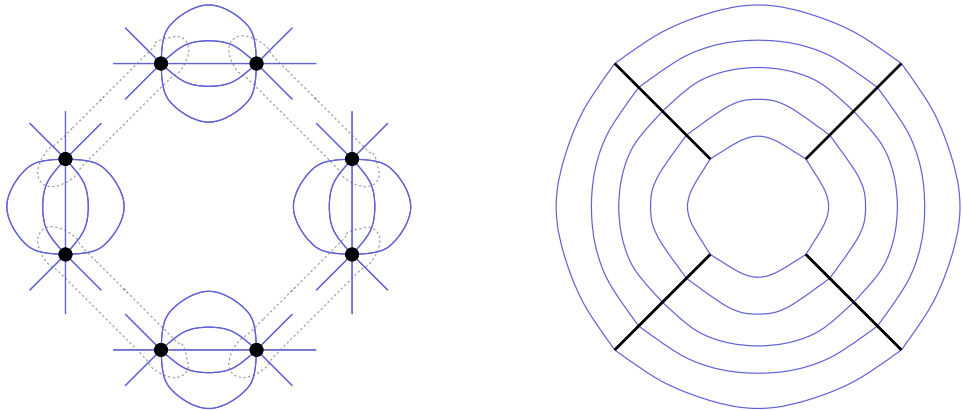}
\caption{\label{fig:T3} At left, we see four copies of the tight 
$S^1 \times S^2$, represented by open books with two binding components
and cylindrical pages.  For each dotted oval surrounding two binding
components, we construct the contact fiber sum to produce the manifold
at right, containing four special pre-Lagrangian tori (the black line segments)
that separate regions foliated by cylinders.  The result is the tight $3$--torus
$(T^3,\xi_2)$.  In general, one can construct $(T^3,\xi_n)$ from $2n$ copies of
the tight $S^1 \times S^2$.}
\end{figure}

By the above example, every contact structure on $T^3$ is partially planar.
In fact, other than the standard torus $(T^3,\xi_1)$, all contact $3$--tori
also have Giroux torsion, thus $\xi_1$ is the only convex fillable contact 
structure on~$T^3$.  Theorem~\ref{thm:consequence} therefore implies
that every contact type embedding of $T^3$ into a closed symplectic
$4$--manifold separates (and the induced contact structure must be~$\xi_1$).
This result is not true for embeddings of weak contact type:
in fact all of the tight tori $(T^3,\xi_n)$ admit weak 
symplectic semifillings with disconnected boundary \cite{Etnyre:private}, 
and thus by the construction in
Example~\ref{ex:EtnyresExample}, they also admit non-separating weakly
contact type embeddings.  

Recall however that if $(W,\omega)$ is a 
weak filling of $(M,\xi)$ and $M$ is a homology
$3$--sphere, then $\omega$ can always be deformed in a collar neighborhood of
$\p W$ to produce a convex filling of $(M,\xi)$; see for instance
\cite{Geiges:book}*{Lemma~6.5.5}.  Thus our results have corresponding
versions for weakly contact hypersurfaces that are homology $3$--spheres.
For example, since the only tight contact structure on $S^3$ is planar,
every weakly contact type embedding of $S^3$ into a closed symplectic
$4$--manifold must separate.

Here is a more general example that also implies the observation made 
above about the $3$--torus.  Let 
$$
\Sigma = \Sigma_+ \cup_\Gamma \Sigma_-
$$
denote any closed 
oriented surface obtained as the union of two nonempty surfaces with
boundary $\Sigma_\pm$ along a multicurve $\Gamma \subset \Sigma$.
By results of Giroux \cite{Giroux:cercles} and
Honda \cite{Honda:tight2}, the manifold $M_\Gamma := S^1 \times \Sigma$
admits a unique (up to isotopy) $S^1$--invariant contact structure
$\xi_\Gamma$ which makes $\Gamma$ the dividing set on 
$\{\text{const}\} \times \Sigma$.  We claim that $(M_\Gamma,\xi_\Gamma)$
is partially planar whenever there exists a connected component of $\Sigma\setminus\Gamma$
having genus zero.  Indeed, for any connected component $\Sigma_0 \subset
\Sigma\setminus\Gamma$, the closure of $S^1 \times \Sigma_0$ may be viewed
as an open book with page $\Sigma_0$ and trivial monodromy, blown up at all
its binding circles; the entirety of $(M_\Gamma,\xi_\Gamma)$ can thus be
obtained by attaching these blown up open books.  (The tight $3$--tori
arise from the case where $\Sigma \cong T^2$ and $\Gamma$ is a union of
parallel curves that are primitive in $H_1(T^2)$.)
Moreover, using Etnyre's obstruction \cite{Etnyre:planar} it is easy to construct many examples $(M_\Gamma,\xi_\Gamma)$ which are partially planar (as just explained) but not planar.
Theorem~\ref{thm:consequence} now implies:

\begin{Cor}
\label{cor:S1invariant}
If $\Sigma\setminus\Gamma$ has a connected component of genus zero, then
the $S^1$--invariant contact manifold $(S^1 \times \Sigma,\xi_\Gamma)$
does not admit any non-separating contact type embeddings into closed
symplectic $4$--manifolds.
\end{Cor}

Finally, the following demonstrates that in some settings
where non-separating hypersurfaces can be embedded \emph{smoothly}, 
they can never be contact type. In contrast to Theorem \ref{thm:consequence}, here the assumptions are on the ambient symplectic 4-manifold and not the contact manifold.

\begin{Thm}\label{thm:rationalRuled}
If the closed and connected symplectic 4-manifold $(W,\omega)$ contains a 
symplectically embedded sphere $S \subset W$ with self-intersection number
$S \bullet S \ge 0$, then every closed contact type hypersurface in $W$ is
separating.
\end{Thm}

The reason for this is closely related to McDuff's results
\cite{McDuff:rationalRuled}, which imply that $(W,\omega)$ in this situation
is always rational or ruled (up to symplectic blowup).  In fact, the case where 
$S \bullet S > 0$
follows immediately from \cite{McDuff:rationalRuled}, which shows that $W$ is 
then a blowup of either $S^2 \times S^2$ or $\C P^2$ and thus 
simply connected, so it does not admit non-separating hypersurfaces at all
(contact or otherwise).  The case $S \bullet S = 0$ is more interesting:
the key fact here is that one can choose a
compatible almost complex structure $J$ for which any given contact hypersurface
$M \subset W$ is $J$--convex, and $W$ is foliated by a family of
embedded $J$--holomorphic spheres (possibly including some isolated nodal
spheres unless $(W,\omega)$ is minimal).  If $M$ does not separate, then
there exists a connected infinite cover $(\widetilde{W},\tilde{J})$ of
$(W,J)$, constructed by gluing together infinitely many copies of
$W \setminus M$ in a sequence.  Now the $J$--holomorphic spheres in
$W$ lift to $\widetilde{W}$ and form a foliation, which must include a
$J$--holomorphic sphere that touches a lift of $M$ tangentially from below,
violating $J$--convexity.  That's a quick sketch of the proof---we'll give an alternative
proof in \S\ref{sec:proofs} that fits into a usefully generalized context and
doesn't assume the results of \cite{McDuff:rationalRuled}.
There are obvious examples of
smoothly embedded non-separating hypersurfaces in ruled surfaces,
e.g.~$\ell \times S^2 \subset \Sigma \times S^2$, where $\Sigma$ is any closed
oriented surface of positive genus
and $\ell \subset \Sigma$ is a non-separating closed curve.  It follows
that a hypersurface isotopic to this one is never contact type.

\subsection{Open questions}

Let $\Contact$ denote the collection of closed $3$--manifolds with positive,
cooriented contact structures, and consider the inclusions
$$
\Nonsep \subsetneq \Embed \subsetneq \Contact,
$$
where $\Embed$ denotes all $(M,\xi) \in \Contact$ that admit a contact type
embedding into {some} closed symplectic manifold, and $\Nonsep$ denotes those
that admit a non-separating embedding.  The results stated in 
\S\ref{subsec:main} imply that both inclusions are proper.

Observe that if $(M,\xi)$ is convex fillable then it is also in
$\Embed$, since a filling can always be capped to produce a closed
symplectic manifold.  Conversely, if $(M,\xi)$ admits a \emph{separating} 
contact type embedding, then it is fillable.  While 
the same is not strictly true for a non-separating embedding, the 
construction depicted in 
Figure~\ref{fig:infinitechain} of \S\ref{sec:proofs} can be viewed as a
filling that is noncompact but \emph{geometrically bounded},
which makes it a good setting for $J$--holomorphic curves.  In this context,
any filling obstruction that involves $J$--holomorphic curves can
also serve as an obstruction to non-separating contact embeddings
(cf.~Corollary~\ref{cor:torsion}), thus implying that $(M,\xi) \not\in \Embed$.
This motivates the conjecture that, in fact, $\Embed$ is the same
as the set of convex fillable contact $3$--manifolds.

\begin{Conj}
If $(M,\xi)$ is not convex fillable, then it admits no contact type
embeddings into any closed symplectic manifold.
\end{Conj}

Equivalently, this would mean there is no contact $3$--manifold that admits
\emph{only} non-separating contact type embeddings.

A more ambitious conjecture would arise from Example~\ref{ex:EtnyresExample},
which is the only method we are yet aware of for constructing non-separating
contact embeddings: $(M,\xi) \in \Nonsep$ whenever it admits a convex
semifilling with disconnected boundary.  The latter class of contact
manifolds is evidently somewhat special, and
one wonders whether it might be equal to $\Nonsep$.

\begin{Question}
Is there a contact $3$--manifold that admits a non-separating contact type
embedding but not a convex semifilling with disconnected boundary?
\end{Question}

Finally, observe that while Theorem~\ref{thm:rationalRuled} rules out the
existence of a non-separating contact hypersurface $(M,\xi) \subset (W,\omega)$
if $(W,\omega)$ is rational or ruled, it still allows the possibility that 
$(M,\xi) \in \Nonsep$ but admits a \emph{separating} embedding into
$(W,\omega)$.  There is some reason to suspect that this could still never
happen.  There are indeed cases where the existence of a contact
embedding of $(M,\xi)$ into some particular symplectic manifold implies
$(M,\xi) \not\in \Nonsep$, e.g.~this is true if $(M,\xi) \hookrightarrow
(\R^4,\omega_0)$.
Moreover, the simplest known example of a manifold in $\Nonsep$, the
unit cotangent bundle of a higher genus surface, has been shown by
Welschinger \cite{Welschinger:effective}
to admit no contact type embeddings into rational or ruled
symplectic $4$--manifolds.

\begin{Question}
Is there a contact $3$--manifold that admits a contact type embedding into
some rational/ruled symplectic $4$--manifold and also admits a
non-separating contact type embedding into some other closed symplectic
manifold?
\end{Question}

\section{Pseudoholomorphic curves in symplectizations}
\subsection{Technical background}
\label{sec:technical}

In this section we collect a number of important technical definitions.
A positive {contact form} on a $3$--manifold $M$ is a $1$--form 
$\lambda$ for which $\lambda\wedge d\lambda > 0$. 
 The $2$-plane distribution 
$\xi:=\ker\lambda$ is then a contact structure.  The equations  $\iota_{X_{\lambda}}d\lambda=0$ and $ \lambda(X_{\lambda})=1$ 
uniquely  determine a vector field $X_{\lambda}$,
called the {Reeb vector field} associated to $\lambda$.  Since $X_{\lambda}$ is everywhere 
transverse to $\xi$, one obtains a splitting $TM=\R X_{\lambda} \oplus \xi$.
Moreover,  $(\xi,d\lambda|_{\xi})$ is a symplectic vector bundle, and the flow
of $X_\lambda$ preserves $\lambda$, hence also $(\xi,d\lambda|_{\xi})$.

A {periodic Reeb orbit} of period $T>0$ for a contact form $\lambda$ is a 
smooth map $\gamma:\R/T\Z \to M$ satisfying $\dot{\gamma}(t)=X_{\lambda}(\gamma(t))$.  
We identify all possible reparametrizations 
$t \mapsto \gamma(t+\hbox{const})$.  A Reeb orbit is called simply covered  
if it has degree $1$ onto its image, i.e. it is an embedding.  {If $\gamma$
covers a simply covered orbit with period~$\tau > 0$, we call $\tau$ the
minimal period of~$\gamma$.}

Since the Reeb flow preserves the symplectic vector bundle 
$(\xi,d\lambda|_{\xi})$, linearizing about a periodic orbit $\gamma$ determines a 
symplectic linear map $d\phi_T(p):\xi_p \to \xi_p$
for each $p$ in the image of $\gamma$.  Then $\gamma$ is said to be 
{nondegenerate} if $1$ is not an eigenvalue of this map; this condition
is independent of the point~$p$.  
More generally, an orbit $\gamma$ of period $T$ is {Morse-Bott} 
if it lies in a submanifold $N\subset M$ foliated by $T$--periodic orbits, 
such that the $1$--eigenspace of $d\phi_T(p)$ is precisely $T_p N$.
We then call $N$ a {Morse-Bott submanifold}.  A contact form 
$\lambda$ is said to be {nondegenerate} if all of its periodic Reeb orbits are 
nondegenerate, and Morse-Bott if every periodic orbit belongs to a
Morse-Bott submanifold.  

Given a symplectic trivialization $\Phi$ of $(\xi,d\lambda)$ along a
$T$--periodic orbit $\gamma$, the linearized flow $d\phi_t(p)$ for
$t \in [0,T]$ defines a continuous family of symplectic matrices, which
has a well defined Conley-Zehnder index if $\gamma$ is nondegenerate: we
denote this index by $\CZ^\Phi(\gamma) \in \Z$.

It is convenient also to express this in terms of asymptotic operators:
associated to any $T$--periodic Reeb orbit $\gamma$ is a linear operator
$\mathbf{A}_\gamma:\Gamma(\x^*\xi) \to \Gamma(\x^*\xi)$,
where $\x:\R/\Z\to M$ is the reparametrization $\x(t):=\gamma(Tt)$.
If $\nabla$ is a 
symmetric connection on $TM$ and $J$ is a complex 
structure on $\xi\rightarrow M$ compatible with the symplectic structure
$d\lambda|_\xi$, then $\mathbf{A}_\gamma$ can be defined on smooth sections by 
\[
        \mathbf{A}_\gamma\eta=-J(\nabla_t\eta - T\nabla_\eta X_\lambda).
\]
This expression is independent of the choice of connection.  Choosing
a unitary trivialization $\Phi$ of $\x^*\xi$, $\mathbf{A}_\gamma$ 
is identified with the operator
\begin{equation}
\label{eqn:asympTriv}
C^\infty(S^1,\R^2) \to C^\infty(S^1,\R^2) : \eta \mapsto -J_0 
\frac{d}{dt}\eta - S\cdot \eta,
\end{equation}
where $S(t)$ is some smooth loop of symmetric $2$--by--$2$ matrices.
Thus the equation $\mathbf{A}_\gamma \eta = 0$ defines a linear Hamiltonian
flow, and one can show that 
the resulting family of symplectic matrices matches the family
obtained from $d\phi_t(p)$.  It follows that $\mathbf{A}_\gamma$ has
trivial kernel if and only if $\gamma$ is nondegenerate, and we can
use the linear Hamiltonian flow determined by \eqref{eqn:asympTriv} 
to define an integer $\CZ^\Phi(\mathbf{A}_\gamma)$, which matches
$\CZ^\Phi(\gamma)$.  The advantage of this definition is that it does not
reference the orbit directly, but makes sense
for any operator that takes the form of \eqref{eqn:asympTriv} in the
trivialization: in particular we can define 
$\CZ^\Phi(\mathbf{A}_\gamma - c) \in \Z$
whenever $c \in \R$ is not an eigenvalue of $\mathbf{A}_\gamma$, even if
$\gamma$ is degenerate.  For this we will use the shorthand notation
$$
\CZ^\Phi(\gamma - c) := \CZ^\Phi(\mathbf{A}_\gamma - c).
$$
We now recall some of the important spectral properties of asymptotic operators.  
For more details and proofs we refer to \cite{HWZ:props2}. 

$\mathbf{A}_\gamma$ extends to an unbounded self-adjoint operator 
on the complexified Hilbert space $L^2(\x^*\xi)$; its spectrum
$\sigma(\mathbf{A}_\gamma)$ consists of
real eigenvalues of multiplicity at most $2$ that accumulate only at infinity.  
Generalizing the statement above about nondegeneracy,
if $\gamma$ belongs to a Morse-Bott submanifold of dimension
$n \in \{1,2,3\}$, then the 
$0$--eigenspace of $\mathbf{A}_\gamma$ is $(n-1)$--dimensional.

Geometric properties of the eigenspaces are closely related to the Conley-Zehnder 
index.  Indeed, any eigenfunction $\eta$ of $\mathbf{A}_\gamma$ has a well
defined winding number $\wind^\Phi(\eta) \in \Z$ relative to the trivialization,
which is independent of the choice of $\eta$ in its eigenspace.
Thus we may 
speak of the winding number $\wind^\Phi(\mu)\in\Z$ for each eigenvalue
$\mu \in \sigma(\mathbf{A}_\gamma)$, and it turns out that the map
$\sigma(\mathbf{A}_\gamma) \to \Z : \mu \mapsto \wind^\Phi(\mu)$ is
non-decreasing and attains every value exactly twice (counting multiplicity).
The following integers
\begin{align*}
 \alpha_{-}^\Phi(\gamma)&:=\max\{\wind^\Phi(\mu)\mid \mu<0\text{ is an eigenvalue of }\mathbf{A}_\gamma\} \\
 \alpha_{+}^\Phi(\gamma)&:=\min\{\wind^\Phi(\mu)\mid \mu>0\text{ is an eigenvalue of }\mathbf{A}_\gamma\} 
\end{align*}
are therefore determined by the eigenfunctions with eigenvalues closest to 
$0$ that are negative and positive respectively.  The number
$p(\gamma) := \alpha^\Phi_+(\gamma) - \alpha^\Phi_-(\gamma)$ is called the 
parity of $\gamma$; it is independent of~$\Phi$ and necessarily equals
$0$ or~$1$ if $\gamma$ is nondegenerate.  More generally, we can replace
$\mathbf{A}_\gamma$ by $\mathbf{A}_\gamma - c$ for some $c \in \R$ and
similarly define $\alpha^\Phi_\pm(\gamma - c)$ and 
$p(\gamma - c)$; then if $c \not\in\sigma(\mathbf{A}_\gamma)$,
a result in \cite{HWZ:props2} implies the relation
\begin{equation}\label{eqn:CZwinding}
\CZ^\Phi(\gamma - c) = 2\alpha_{-}^\Phi(\gamma - c) + p(\gamma - c)
= 2\alpha_{+}^\Phi(\gamma - c) - p(\gamma - c).
\end{equation} 

Observe that every Morse-Bott submanifold of dimension~$2$ admits a nonzero
vector field and is thus either a torus or a Klein bottle.
The following characterization of Morse-Bott tori is a simple consequence of
the spectral properties of $\mathbf{A}_\gamma$ 
(cf.~\cite{Wendl:automatic}*{Prop.~4.1}).

\begin{Prop}
\label{prop:MBtori}
Suppose $\gamma$ is a Morse-Bott periodic orbit of $X_\lambda$ belonging to a
Morse-Bott submanifold $N \subset M$ diffeomorphic to~$T^2$.  Then the
Morse-Bott property is satisfied for all covers of all orbits in~$N$, and
they all have the same minimal period.
\end{Prop}

We will also need a relative version of the standard genericity result for
nondegenerate contact forms.

\begin{Lemma}
\label{lemma:nondegeneracy}
Suppose $N \subset M$ is a union of $2$--tori which are Morse-Bott submanifolds
for some contact form~$\lambda_0$.  Then for any $T_0 > 0$, there exists an
arbitrarily small perturbation $\lambda$ of $\lambda_0$ such that
$\lambda=\lambda_0$ on a neighborhood of $N$ and every periodic orbit of
$X_{\lambda}$ with period less than $T_0$ is Morse-Bott.
\end{Lemma}
\begin{proof}
Since all orbits in $N$ are Morse-Bott (including all multiple covers,
due to Prop.~\ref{prop:MBtori}), for any $T_0 > 0$ we can find an open 
neighborhood $\U$ of $N$ such that $\overline{\U} \setminus N$ contains 
no periodic orbits with
period less than~$T_0$.  By Theorem~\ref{thm:nondegeneracy} in the appendix,
one can then find a generic small perturbation
of $\lambda_0$ with support in $M \setminus \U$ so that all orbits 
passing through $M \setminus \overline{\U}$ are nondegenerate. 
\end{proof}

We now recall the basic notions of holomorphic curves in symplectizations and
their asymp\-totic properties.
The {symplectization} of a contact manifold $(M,\xi=\ker\lambda)$ is the product space 
$\R\times M$ equipped with the exact symplectic form $d(e^a\lambda)$, where 
$a:\R\times M\to\R$ refers to the $\R$ coordinate.  An almost complex structure $J$ 
on the symplectization is said to be {admissible} if it is $\R$--invariant, 
restricts to the {symplectic vector bundle} $(\xi,d\lambda)$ as a compatible complex 
structure, and satisfies  $J\partial_a=X_\lambda$.
Any admissible $J$ tames the symplectic form $d(e^a\lambda)$, and more generally tames every symplectic form 
$d(\varphi\lambda)$ where $\varphi:\R\to(0,\infty)$ is smooth with $\varphi'>0$.  

A  \emph{pseudoholomorphic} (or \emph{$J$--holomorphic} or simply \emph{holomorphic}) 
curve from a punctured Riemann surface 
$(\dot{\Sigma},j)$, into an almost complex manifold $(W,J)$ is a solution 
$u:\Sigmadot\to W$ to the nonlinear Cauchy-Riemann equation $Tu\circ j=J(u)\circ Tu$.
Here we take $\dot{\Sigma}:=\Sigma\setminus\Gamma$ for some finite 
set of points $\Gamma\subset\Sigma$, where $(\Sigma,j)$ is a closed 
connected Riemann surface.

For the rest of this section, let us consider only the case where the target is the symplectization 
of $(M,\lambda)$, and $J$ is an admissible almost complex structure on $\R\times M$.  
The simplest case of a punctured $J$--holomorphic curve in this setting is the
so-called trivial cylinder 
$$
u : S^2\setminus\{0,\infty\} \cong \R \times S^1 \to \R \times M : (s,t) \mapsto
(Ts,\gamma(Tt)),
$$
where $T > 0$ and $\gamma$ is any $T$--periodic Reeb orbit.
Following \cite{Hofer:weinstein,SFTcompactness}, the energy of 
a $J$--holomorphic curve $u:\Sigmadot\to \R\times M$ can be
defined as follows.  Fix any constant $C>0$, and let  
\begin{equation}\label{equ:energyinsymplectization}
        E(u):=\sup_{\varphi\in\mathcal{T}}\int_{\Sigmadot}u^*d(\varphi\lambda)
\end{equation}
where $\mathcal{T}$ is the set of smooth maps 
$\varphi:\R\to (0,C)$ with $\varphi'>0$.
Since $J$ is compatible with $d(\varphi\lambda)$ for
all $\varphi \in \tT$, the integrand in
\eqref{equ:energyinsymplectization} is always nonnegative, thus $u$ is
constant if and only if its energy vanishes.  Observe that the integrand
of $\int_{\Sigmadot} u^*d\lambda$ is also nonnegative, and this integral is
finite if $u$ has finite energy: it vanishes identically
if and only if $u$ is a branched cover of a trivial cylinder.

\begin{Def}\label{def:finiteenergycurve}
We will say that $u:\Sigmadot\to \R\times M$ is a 
{finite energy $J$--holomorphic curve} if it is proper 
and $E(u)<\infty$.
\end{Def}
Note that properness only fails when there exist punctures having 
neighborhoods which are mapped into a compact set, in which case these
punctures can be removed by Gromov's removable singularity theorem.
Since $d(\varphi\lambda)$ is exact, Stokes' theorem implies that not all 
punctures are removable unless $u$ is constant.

Let us recall now the behaviour of a finite energy $J$--holomorphic curve $u :\Sigmadot\to \R\times M$ in the 
neighborhood of a puncture.  Each puncture $z\in\Gamma$ has a neighborhood on which the $\R$--value of $u$ tends to 
$+\infty$ or $-\infty$, and we say that $z$ is a {positive/negative puncture}
respectively.  Denote the resulting partition into positive and negative punctures by $\Gamma=\Gamma^+\cup\Gamma^-$.
Restricting to a neighborhood of a puncture, we obtain a curve whose domain is
the punctured closed disc, which is biholomorphic to both $Z^+:=[0,\infty)\times S^1$ and $Z^-:=(-\infty,0]\times S^1$ with the standard complex structure.
It is convenient to choose the domain of the restricted curve to be $Z^+$ or $Z^-$ 
for $z\in\Gamma^+$ or $z\in\Gamma^-$ respectively, and we will write $u:Z^\pm\to \R\times M$.
It was shown by Hofer in \cite{Hofer:weinstein} that for any sequence $|s_k|\to\infty$, there 
exists a subsequence such that $u(s_k,\cdot)$ converges in 
$C^\infty(S^1,M)$ to $\gamma(T\cdot)$, where $\gamma$ is a $T$--periodic Reeb
orbit for some $T > 0$.  We say in this case that $u$ is asymptotic to $\gamma$,
and $\gamma$ is an asymptotic orbit of~$u$.

In the following statement, we choose any $\R$--invariant connection on 
$\R\times M$ to define the exponential map, and use the term \emph{asymptotically trivial
coordinates} to refer to a diffeomorphism $(\sigma,\tau) : Z^\pm \to Z^\pm$
such that $\sigma(s,t) - s$ and $\tau(s,t) - t$ approach constants
as $|s| \to \infty$ and their derivatives of all orders decay to zero.

\begin{Thm*}[\cite{HWZ:props1,HWZ:props4,Mora}]  
Suppose $u:Z^\pm\to \R\times M$ has finite energy and is asymptotic to a Morse-Bott Reeb 
orbit $\gamma$ of period $T>0$.  Then there exist asymptotically
trivial coordinates $(\sigma,\tau)$ such that for sufficiently large $|\sigma|$, either
$u(\sigma,\tau) = (T\sigma,\gamma(T\tau))$ or
\begin{equation}\label{eqn:asymptoticformula}
                u(\sigma,\tau)=\exp_{(T\sigma,\gamma(T\tau))}
\left[e^{\mu \sigma}(e_\mu(\tau) + r(\sigma,\tau))\right],
\end{equation}
where $e_\mu$ is an eigenfunction of $\mathbf{A}_\gamma$ with eigenvalue 
$\mu \in \sigma(\mathbf{A}_\gamma)$ such that $\pm\mu < 0$, and the
``remainder'' term
$r(\sigma,\tau) \in \xi_{\gamma(T\sigma)}$ decays to zero uniformly
with all derivatives as $|\sigma| \to \infty$.
\end{Thm*}
\begin{Def}
When \eqref{eqn:asymptoticformula} holds, we call $e_\mu$ the asymptotic
eigenfunction of $u$ at the puncture, and say that $u$ has
transversal convergence rate~$|\mu|$.
In the case where
$u(\sigma,\tau) = (T\sigma,\gamma(T\tau))$, we define the asymptotic 
eigenfunction to be~$0$ and the transversal convergence rate to be~$\infty$.
\end{Def}
Observe that the asymptotic eigenfunction $e_\mu$ is
determined uniquely once a parametrization of $\gamma$ is fixed.
We know also from the monotonicity of winding numbers that 
$\wind^\Phi(e_\mu)\leq\alpha_-^\Phi(\gamma)$ if the puncture is positive,
and $\wind^\Phi(e_\mu) \ge \alpha_+^\Phi(\gamma)$ if it is negative.

Let $\pi_\lambda : TM \to \xi$ denote the natural projection
with respect to the splitting $TM = \R X_\lambda \oplus \xi$ and suppose
$u = (u_\R,u_M) : \dot{\Sigma} \to \R\times M$ is a finite energy
$J$--holomorphic curve.  Then
the composition $\pi_\lambda \circ Tu_M$ defines a section of the bundle of
complex linear homomorphisms $(T\dot{\Sigma},j) \to (u^*\xi,J)$.  As shown
in \cite{HWZ:props2}, this section satisfies a linear Cauchy-Riemann type
equation, and thus is either trivial or has a discrete set of zeros, all of 
positive order.  The former holds if and only if any asymptotic eigenfunction
of $u$ vanishes, in which case they all do: then $\int_{\Sigmadot} u^*d\lambda
= 0$ and $u$ is a branched cover of a trivial cylinder.  Otherwise,
\eqref{eqn:asymptoticformula} implies that $\pi_\lambda \circ Tu_M$ has
finitely many zeros, and we denote the algebraic count of these by
$$
\windpi(u) \in \Z.
$$
Clearly $\windpi(u) \ge 0$, with equality if and only if
$u_M : \Sigmadot \to M$ is an immersion.

\subsection{Property~$(\star)$ and the main results}
\label{subsec:star}

We now use holomorphic curves to define two technical conditions on contact 
manifolds which imply the results stated in \S\ref{sec:introduction}. Property~$(\star)$ and its weak version, introduced below, will serve as 
obstructions to the
existence of non-separating contact embeddings.  They are
implied by each of
the contact topological assumptions mentioned in Theorem~\ref{thm:consequence},
and in fact are more general (see also \cite{Wendl:fiberSums}).

\begin{Def}
\label{def:star}
A closed three-dimensional contact manifold $(M,\xi)$ satisfies 
property $(\star)$ if there exists a contact form $\lambda$ with
$\ker\lambda = \xi$ and an admissible $\R$--invariant almost complex
structure $J$ on the symplectization $\R\times M$, which admits
a finite energy $J$--holomorphic punctured sphere
$$
u = (u_\R,u_M) : \dot{\Sigma} = S^2 \setminus \{z_1,\ldots,z_N\} \to 
\R \times M
$$ 
with the following properties:
\begin{enumerate}
\item $u_M$ is an embedding, and the closure of
$u_M(\dot{\Sigma}) \subset M$ is an embedded surface whose oriented boundary
is a union of Reeb orbits, called the ``asymptotic orbits'' of~$u$.
\item {Each asymptotic orbit of $u$ is nondegenerate or Morse-Bott.}
\item If $T_1,\ldots,T_N$ are the periods of the asymptotic orbits of $u$,
then every Reeb orbit not in the same Morse-Bott submanifold with one of these
has period strictly greater than $T_1 + \ldots + T_N$.
\item $u$ has no asymptotic orbit that is nondegenerate with
Conley-Zehnder index zero, relative to the natural trivialization
determined by the image of $u_M$ near the puncture.
\item
{If any asymptotic orbit of $u$ belongs to a $2$--dimensional Morse-Bott
manifold $N \subset M$ disjoint from $u_M(\dot{\Sigma})$, then $N$ is a torus
and contains no other asymptotic orbits of~$u$.}
\end{enumerate}
\end{Def}

\begin{Rmks*} \ 
\begin{itemize}
\item
The fact that Reeb orbits comprise the \emph{oriented} boundary of
$\overline{u_M(\dot{\Sigma})}$ implies that every puncture 
of $u$ is positive.  Moreover, each puncture is asymptotic to a
distinct Reeb orbit, which is simply covered.
\item 
The asymptotic formula \eqref{eqn:asymptoticformula} implies that on
each cylindrical end of $\dot{\Sigma}$, $u_M$ does not intersect the
corresponding asymptotic orbit, thus it defines a natural trivialization
of $\xi$ along this orbit.  One can then show
(cf.~\eqref{eqn:CZwinding}) that relative to this trivialization, the
orbit always has nonnegative Conley-Zehnder index if it is nondegenerate---thus
our definition requires this index to be anything
strictly larger than the minimum possible value.
\end{itemize}
\end{Rmks*}

\begin{Def}\label{def:weak_star}
We say that a closed three-dimensional contact manifold $(M,\xi)$ 
satisfies weak property $(\star)$ if there is a symplectic cobordism 
{$(W,\omega)$} from $(M,\xi)$
to a contact manifold $(M',\xi')$, {such that either $(W,\omega)$ contains
a symplectically embedded sphere of nonnegative self-intersection number or
$(M',\xi')$ satisfies property $(\star)$.}
\end{Def}

For example, $(M,\xi)$ satisfies weak property~$(\star)$ if it
{admits a symplectic cap containing a nonnegative symplectic sphere,
or if it} can be made
to satisfy property~$(\star)$ after a sequence of contact $(-1)$--surgeries 
or connected sum operations.
Obviously property~$(\star)$ implies weak property~$(\star)$, and it's
plausible that the converse may also be true, though this is presumably
hard to prove.

We can now state some more technical results that imply 
Theorem~\ref{thm:consequence}.  These will be proved in 
\S\ref{sec:proofs}, using the machinery of \S\ref{sec:compactness}.

\begin{Thm}\label{thm:separating}
Let $(W,\om)$ be a closed and connected
symplectic 4-manifold which contains a closed 
contact type hypersurface $M\subset W$ satisfying weak property $(\star)$. 
Then $M$ separates~$W$.
\end{Thm}

\begin{Thm}\label{thm:multiple}
Let $(W,\om)$ be a compact and connected symplectic 4-manifold with convex
boundary containing a connected component $M \subset \p W$ that satisfies 
weak property $(\star)$.  Then $\p W$ is connected.
\end{Thm}

\begin{Thm} \label{thm:separating_general}
Let $(W,\om)$ be a compact and connected 4-manifold with convex boundary 
$(M,\xi)$ satisfying weak property $(\star)$. Then any closed contact 
type hypersurface $H$ in $W\setminus M$ separates $W$ into a convex filling
of $H$ and a symplectic cobordism from $H$ to $M$.  In particular,
$H$ also satisfies the weak $(\star)$ property.
\end{Thm}

\begin{Rmk}
A compact connected symplectic manifold with convex boundary can never
contain a symplectic sphere of nonnegative self-intersection.  This follows
easily from the arguments we will use to prove the above results: otherwise
one would find a family of embedded holomorphic spheres foliating the
positive end of the symplectization of the convex boundary,
and thus violating the maximum principle.
\end{Rmk}

\begin{Rmk}
Note that property~$(\star)$ depends only on the contact structure:
we do \textbf{not} assume in any of these theorems that 
the contact form induced on $M$ by a Liouville vector field is the same one 
which appears in Definition~\ref{def:star}.
\end{Rmk}

We will show in \S\ref{sec:torsion} that any contact manifold
$(M,\xi)$ with Giroux torsion satisfies Property~$(\star)$.  It turns out
that this is also true for a contact fiber sum of open books
$(M,\xi) = \#_\Phi (M_i,\xi_i)$ whenever any of the summands $(M_i,\xi_i)$
is planar.  This follows from
an important relationship between open books and holomorphic curves:
namely, it is shown in \cites{Abbas:openbook,Wendl:openbook} that if the open 
book on $(M_i,\xi_i)$ is planar, one can take its pages to be projected images
of embedded index~$2$ holomorphic curves.  A minor variation on this 
construction in \cite{Wendl:openbook2} extends it to the 
blown up manifold $(\widehat{M}_i,\hat{\xi}_i)$: 
the difference here is that each holomorphic page is asymptotic to a 
different orbit in a Morse-Bott family foliating the boundary.  
Moreover, one can easily arrange the contact form in this construction
so that all the asymptotic orbits are either elliptic or Morse-Bott and
have much smaller period than any other Reeb orbit in
$\#_\Phi (M_i,\xi_i)$.  It follows that $\#_\Phi (M_i,\xi_i)$
satisfies property~$(\star)$ if any of its constituent open books is planar.

\section{Giroux torsion}\label{sec:torsion}

Following a construction in (\cite{Wendl:fillable}) but being more careful 
about periods, we now establish the following.
\begin{Prop}\label{prop:Torsion implies prop star}
Let $(M,\xi)$ be a closed contact manifold having Giroux torsion.  Then 
$(M,\xi)$ satisfies property $(\star)$.  
\end{Prop}
\begin{proof}

By definition, Giroux torsion means that $(M,\xi)$
contains a subset $T$ that can be identified
with a thickened torus $S^1 \times S^1 \times [0,1]$, on which $\xi$
has the form
\begin{equation}\label{eqn:contactstructureontorus}
\xi = \ker\left[\cos(2\pi\theta)dx + \sin(2\pi\theta)dy\right]
\end{equation}  
in coordinates $(x,y,\theta) \in S^1 \times S^1 \times [0,1]$.
Let us assume $\xi = \ker\lambda$
for some contact form $\lambda$ that is Morse-Bott outside of $T$, and
in $T$ has the form
$\lambda = f(\theta)\ dx + g(\theta)\ dy$ for smooth functions
$f,g:[0,1] \to \R$ with
$$
\gamma(\theta) := (f(\theta),g(\theta)) = h(\theta) e^{2\pi i \theta} \in \R^2,
$$
where $h(\theta) > 0$ and $h(\theta) = 1$ for $\theta$ near $0$ and~$1$.
The path $\gamma$ is thus closed and 
bounds a star-shaped region in $\R^2$, and we will
show that $\lambda$ has the desired properties if $\gamma$ bounds a
suitably oblong oval.

The Reeb vector field of $\lambda$ on $T$ is given by 
\begin{equation}\label{equ:formulaforreebfield}
        X_\lambda=\frac{1}{D(\theta)}(g'(\theta)\partial_x - f'(\theta)\partial_y),
\end{equation}
where $D(\theta) := f(\theta) g'(\theta) - f'(\theta) g(\theta) > 0$.
Since this has no $\partial_\theta$ component, each torus 
$N(\theta_0):=\{(x,y,\theta_0)\ |\ (x,y)\in S^1\!\times\! S^1\}\subset T$
is invariant under the Reeb flow. Moreover, the Reeb flow on each
$N(\theta)$ is linear and has closed orbits if and only if
$dx(X_\lambda)/dy(X_\lambda) \in \Q \cup \{\infty\}$.  From 
\eqref{equ:formulaforreebfield}, this ratio is
$-g'(\theta) / f'(\theta) = -\slope(\gamma'(\theta))$, so
$N(\theta)$ has closed orbits precisely when $\slope(\gamma'(\theta))$ 
is rational or {infinite}.  In this case \emph{every} orbit in 
$N(\theta)$ is closed and represents the same class in $H_1(N(\theta)) = \Z^2$,
which we will denote by a pair of integers $(p(\theta),q(\theta))$ with  
$\gcd(|p(\theta)|,|q(\theta)|)=1$ and
\begin{equation}\label{equ:slopeequalsqoverp}
        \frac{p(\theta)}{q(\theta)}=-\slope(\gamma'(\theta)) \in
        \Q \cup \{\infty\}.
\end{equation}
Since $d\lambda$ vanishes on $N(\theta)$, all closed simply covered orbits in
$N(\theta)$ have the same period, which we will denote by $\T(\theta) > 0$.
If $\sigma:\R/\Z\to N(\theta)$ parametrizes such an orbit, we compute
\begin{equation}\label{equ:formulaforperiodsontori}
                \T(\theta)=\int_0^1\sigma^*\lambda=p(\theta)f(\theta) + q(\theta)g(\theta).  
\end{equation}

\begin{Lemma}\label{lemma:ellipseisgoodenough}
Fix $\epsilon>0$ small and assume that in addition to the above conditions,
$\gamma(\theta) = h(\theta) e^{2\pi i\theta}$ bounds a convex set symmetric
about both axes, $h(1/4) = h(3/4) = \epsilon$ and $\gamma'(\theta)$ and
$\gamma''(\theta)$ are always linearly independent.
Then:
\begin{enumerate}
 \item $\lambda$ is Morse-Bott.
 \item $X_\lambda=\frac{1}{\epsilon}\partial_y$ on $N(1/4)$ and 
 $-\frac{1}{\epsilon}\p_y$ on $N(3/4)$.
 \item $\T(1/4) = \T(3/4) = \epsilon$, and $\T(\theta) > 1/4$ for all other
 $\theta$ at which $N(\theta)$ has closed orbits.
\end{enumerate}
\end{Lemma}
\begin{proof}
It follows by straightforward computation from the assumption that 
$\gamma'(\theta)$ and $\gamma''(\theta)$
are linearly independent that each $N(\theta)$ with closed orbits
is a Morse-Bott submanifold.
The second claim follows immediately from \eqref{equ:formulaforreebfield}
since symmetry requires $g'(1/4) = g'(3/4) = 0$,
and it is then clear that $\T(1/4) = \T(3/4) = \epsilon$.

To show that all other values of $\theta$ have $\T(\theta) > 1/4$,
observe first that by symmetry, we can always assume
$g'$ and $-f'$ have the same sign as $f$ and $g$ respectively.  Thus
$\sign(p)=\sign(dx(X_\lambda))=\sign(g')=\sign(f)$ and 
$\sign(q)=\sign(dy(X_\lambda))=\sign(-f')=\sign(g)$, so
formula \eqref{equ:formulaforperiodsontori} becomes
\begin{equation}\label{equ:betterformulaforperiodsontori}
                \T(\theta)=|p(\theta)||f(\theta)| + |q(\theta)||g(\theta)|.  
\end{equation}
Let $\Delta$ denote the diamond shaped region in the $xy$--plane for which 
$|x|+|y|\leq 1/2$ (see Figure~\ref{fig:oval}).   We deal separately with two 
cases.  

\begin{figure}[hbt]
\psfrag{Delta}{$\Delta$}
\psfrag{half}{$\frac{1}{2}$}
\psfrag{1}{$1$}
\psfrag{-1}{$-1$}
\psfrag{epsilon}{$\epsilon$}
\psfrag{-epsilon}{$-\epsilon$}
\psfrag{gamma}{$\gamma$}
\includegraphics[scale=.45]{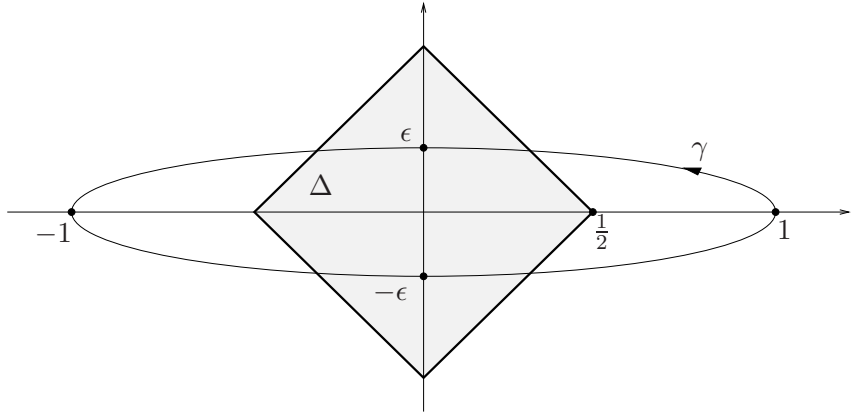}
\caption{\label{fig:oval} The curve $\gamma$ and (shaded) region $\Delta$
in Lemma~\ref{lemma:ellipseisgoodenough}.}
\end{figure}

\textbf{Case $\gamma(\theta)\in \Delta$:}  In this region, outside of the
special values $\theta=1/4,3/4$ we have
$0<|\slope(\gamma'(\theta))| < 2\epsilon$, and by convexity,
$|g(\theta)| > \epsilon / 2$.
With the slope nonzero, it follows from \eqref{equ:slopeequalsqoverp} that 
both $p$ and $q$ are nonzero: in particular $|p|\geq 1$.  Then
from the previous inequality, 
\begin{align*}
                |q|=\frac{|q|}{|p|}|p|=\frac{1}{|\slope(\gamma'(\theta))|}|p|>
                \frac{1}{2\epsilon} |p|\geq \frac{1}{2\epsilon},
\end{align*}
and using \eqref{equ:betterformulaforperiodsontori}, 
$\T(\theta)\geq |q(\theta)||g(\theta)| > \frac{1}{2\epsilon} 
\frac{\epsilon}{2} = 1/4.$

\textbf{Case $\gamma(\theta)\notin \Delta$:}  After verifying explicitly that
$\T(0) = \T(1) = 1$, we can exclude these two cases and assume once more that
both $p(\theta)$ and $q(\theta)$ are nonzero.
Then \eqref{equ:betterformulaforperiodsontori} gives 
\begin{align*}
        \T(\theta)\geq |f(\theta)| + |g(\theta)| > 1/2
\end{align*}
by the definition of~$\Delta$.
\end{proof}

Using the lemma, we can arrange $\lambda$ in $T$ without changing it in
$M \setminus T$ so that $\T(1/4) = \T(3/4) = \epsilon$ is less than half the 
period of every other periodic orbit in~$M$.  Now copying the construction 
in \cite{Wendl:fillable}*{Example~2.11}, we construct a family of 
embedded $J$--holomorphic cylinders in $\R\times T$ that foliate the
region between $N(1/4)$ and $N(3/4)$, each of the form
$$
u : \R \times S^1 \to \R \times M : (s,t) \mapsto 
(\alpha(s) + a_0, x_0, t, \rho(s)),
$$
where $a_0 \in \R$ and $x_0 \in S^1$ are arbitrary constants,
$\alpha : \R \to \R$ is a fixed function that goes to $+\infty$ at both ends and
$\rho : \R \to (1/4,3/4)$ is a fixed orientation reversing diffeomorphism.
Any of these cylinders satisfies the requirements of property~$(\star)$.
\end{proof}

\section{Fredholm theory, intersection numbers and compactness}
\label{sec:compactness}

In this section, assume $(W,\omega)$ is a connected (and possibly noncompact)
symplectic $4$--manifold with convex boundary 
$\p W = M$.  {The boundary need not be connected or nonempty; 
for simplicity we will assume that it is compact,
though we will later be able to relax this assumption.}
Choosing a Liouville vector field $Y$ and a smooth function
$f : M \to \R$, we define a contact form $\lambda$ on $M$ by
$\iota_Y\omega|_M = e^f \lambda$ and denote by $\xi = \ker\lambda$ the
induced contact structure.
We can then use the reverse flow of $Y$ to identify 
a neighborhood of $\p W$ symplectically with a neighborhood of the
boundary of $(\{ (t,m) \in \R \times M\ |\ t \le f(m) \}, d(e^t\lambda) )$.
Thus we can smoothly attach the cylindrical end
$$
E_+ := (\{ (t,m) \in \R \times M\ |\ t \ge f(m) \}
$$
with symplectic form $d(e^t\lambda)$,
forming an enlarged symplectic manifold $(W^\infty,\omega)$ which naturally 
contains $([T,\infty) \times M, d(e^t\lambda))$ for sufficiently large~$T$.

\begin{Assumption} \label{assumptions}
With $(W,\om)$ as described above, assume either of the following:
\begin{enumerate}
\item
$(W,\omega)$ contains a symplectically embedded sphere
$u_0 : S^2 \to W$ with self-intersection number zero.
\item
$(M,\xi)$ satisfies property~$(\star)$.
\end{enumerate}
\end{Assumption}
In the first case, we can define $\dot{\Sigma} := S^2$ with the standard
complex structure,
choose any admissible $\R$--invariant almost
complex structure $J_+$ on $([T,\infty) \times M, d(e^t\lambda))$ and
extend it to an $\omega$--compatible almost complex structure $J$ on
$W^\infty$ such that $u_0$ is (after reparametrization) a $J$--holomorphic
curve.  
In the second case, we can (by appropriate choice of the function
$f$) take $\lambda$ and $J_+$ to be the particular contact form and almost
complex structure arising from Definition~\ref{def:star}, and again
extend $J_+$ to an $\omega$--compatible structure $J$ on~$W^\infty$.
After a sufficiently large $\R$--translation, the $J_+$--holomorphic
curve given by Definition~\ref{def:star} may then be regarded as a
$J$--holomorphic curve
$$
u_0 = (u_\R,u_M) : \dot{\Sigma} \to [T,\infty) \times M \subset W^\infty,
$$
where $\dot{\Sigma} = S^2 \setminus \{z_1,\ldots,z_N\}$ with the
standard complex structure of~$S^2$.

Given any smooth function $\varphi : \R \to (0,\infty)$ that is
monotone increasing and satisfies $\varphi(t) = e^t$ for $t \le T$,
we can define a new symplectic form on $W^\infty$ by
\begin{equation}
\label{eqn:omegaphi}
\omega_\varphi =
\begin{cases}
\omega & \text{ in $W$,}\\
d(\varphi\lambda) & \text{ in $E_+$.}
\end{cases}
\end{equation}
Observe that $J$ is also compatible with $\omega_\varphi$.  
\begin{Def}\label{def:energyCobordism}
The {energy} of a $J$--holomorphic curve $u : \dot{\Sigma} \to W^\infty$
is
$$
E(u) = \sup_{\varphi \in \tT} \int_{\dot{\Sigma}} u^*\omega_\varphi,
$$
where $\omega_\varphi$ is as defined in \eqref{eqn:omegaphi} and
$\tT$ is the set of all smooth functions $\varphi : \R\to (0,\infty)$
that satisfy $\varphi' > 0$, $\varphi(t) = e^t$ for $t \le T$ and
$\sup \varphi \le e^{2T}$. 
\end{Def}
This is equivalent to the
definition of energy given in \cite{SFTcompactness},
in the sense that uniform bounds on either imply uniform bounds on the other.
As in \S\ref{sec:technical}, we will always assume that finite energy
$J$--holomorphic curves in $W^\infty$ are proper and thus have no removable punctures:
then they also satisfy the asymptotic formula \eqref{eqn:asymptoticformula}
and thus have well defined asymptotic eigenfunctions and transversal
convergence rates at each puncture.

Denote by $\mM^*$ the moduli space of all proper, somewhere injective
finite energy $J$--ho\-lo\-mor\-phic 
curves in $W^\infty$, with arbitrary conformal structures on the domains
and any two curves considered {equivalent} if they are
related by a biholomorphic reparametrization that preserves each puncture.  
We assign to $\mM^*$ the natural topology
defined by $C^\infty$--convergence on compact subsets and
$C^0$--convergence up to the ends, and denote by $\mM^*_0 \subset \mM^*$
the connected component containing~$u_0$.
Observe that since $\int u^*\omega_\varphi$ depends only on $\varphi$ and 
the relative homology class represented by $u$, the energy $E(u)$ is 
uniformly bounded for all $u \in \mM^*_0$.

We shall now define special subsets {$\mM^\cc \subset \mM^*$ and $\mM^\cc_0 \subset \mM^*_0$}, consisting 
of $J$--holomorphic curves that satisfy asymptotic constraints.
If $u_0$ has no punctures, we can simply set {$\mM^\cc = \mM^*$ and
$\mM^\cc_0 = \mM^*_0$}.
Otherwise, let us fix the following notation: for each puncture
$z \in \Gamma$ of $u_0$, denote the corresponding asymptotic orbit of $u_0$ by 
$\gamma_z$, with asymptotic operator $\mathbf{A}_z$, 
asymptotic eigenfunction $e_z$ and transversal convergence 
rate~$-\mu_z$, so $\mu_z \in \sigma(\mathbf{A}_z)$.  Choose any unitary
trivialization $\Phi$ for $\xi$ along each of the orbits $\gamma_z$.
We will define a new partition
$$
\Gamma = \Gamma_C \cup \Gamma_U
$$
{in terms of the asymptotic behavior of $u_0$, calling these}
the \emph{constrained} and \emph{unconstrained} punctures
respectively.  Namely, define $z \in \Gamma$ to be in $\Gamma_C$ if and only
if $\gamma_z$ is either nondegenerate or belongs to a Morse-Bott submanifold
$N \subset M$ that intersects $u_M(\Sigmadot)$.

\begin{Lemma}\label{lemma:MBintersect}
If $\gamma_z$ belongs to a Morse-Bott submanifold $N \subset M$ of dimension
at least~$2$, then $N$ intersects $u_M(\Sigmadot)$ if and only if
$\wind^\Phi(e_z) < 0$, where $\Phi$ is the unique trivialization in which
the nontrivial sections in $\ker\mathbf{A}_z$ have zero winding.
\end{Lemma}
\begin{proof}
It is obvious from the asymptotic formula \eqref{eqn:asymptoticformula}
that $u_M$ intersects $N$ if $\wind^\Phi(e_z) < 0$.  To prove the converse,
observe first that since $u_M$ is embedded, it cannot intersect its own
asymptotic orbits.  One then has to show that if $u_0$ intersects any trivial 
cylinder $\R \times \gamma'$ over an orbit $\gamma'$ in $N$, then it also
has an ``asymptotic intersection'' with $\R \times \gamma_z$, which cannot
be true if $\wind^\Phi(e_z) = 0$.  This follows easily from the intersection
theory of punctured holomorphic curves,
see \cites{Siefring:intersection,SiefringWendl} for details.
\end{proof}

\begin{Lemma}\label{lemma:constraints}
For each $z \in \Gamma_C$,
there exists a number $c_z < 0$ such that
$c_z \not\in \sigma(\mathbf{A}_z)$,
$\alpha^\Phi_-(\gamma_z - c_z) = \wind^\Phi(e_z)$ and
$\alpha^\Phi_+(\gamma_z - c_z) = \wind^\Phi(e_z) + 1$.
\end{Lemma}
\begin{proof}
Choose $\Phi$ so that $\wind^\Phi(e_z) = 0$; in the language of
Definition~\ref{def:star}, this is the special trivialization determined by the
asymptotic behavior of $u_M$ near~$z$.  Then
$\alpha^\Phi_-(\gamma_z) \ge 0$, and if $\gamma_z$ is nondegenerate,
\eqref{eqn:CZwinding} implies $\CZ^\Phi(\gamma_z) \ge 0$, with equality
if and only if $\alpha^\Phi_-(\gamma_z) = \alpha^\Phi_+(\gamma_z) = 0$.
The latter is therefore excluded by the condition
$\CZ^\Phi(\gamma_z) \ne 0$ from Definition~\ref{def:star}.  It follows that
if $\mu \in \sigma(\mathbf{A}_z)$ is the largest eigenvalue with
$\wind^\Phi(\mu) = \wind^\Phi(e_z)$, then $\mu < 0$ and we can choose
$c_z$ to be any number slightly larger than $\mu$.

For the case where $\gamma_z$ is Morse-Bott, the fact that $u_M$ intersects
the Morse-Bott submanifold means $0 = \wind^\Phi(e_z) < \wind^\Phi(0)$
due to Lemma~\ref{lemma:MBintersect}.
Thus the eigenvalue $\mu$ defined above is again negative and we can choose
$c_z$ to be slightly larger.
\end{proof}

In the following, let $c_z < 0$ denote the number given by
Lemma~\ref{lemma:constraints} for each constrained puncture $z \in \Gamma_C$, 
and for $z \in \Gamma_U$ set $c_z := \epsilon > 0$ small
enough so that $(0,\epsilon)$ never intersects $\sigma(\mathbf{A}_z)$.

\begin{Def}\label{def:constrained}
The constrained moduli space $\mM^\cc$
consists of all curves {$u \in \mM^*$ having at most $\#\Gamma$ punctures,
which can be identified with a subset of $\Gamma$ in such a way
that at every $z \in \Gamma_C$ that is a puncture of~$u$,
the asymptotic orbit of $u$ is $\gamma_z$,} with transversal
convergence rate strictly greater than $|c_z|$.  Let $\mM^\cc_0 \subset
\mM^\cc$ denote the connected component containing~$u_0$.
\end{Def}

\begin{Prop}\label{prop:embedded}
Every curve $u \in \mM^\cc_0$ is embedded.
\end{Prop}
\begin{proof}
By Definition~\ref{def:star},
each asymptotic orbit for the curves in $\mM^\cc_0$ is either fixed or allowed
to vary in a Morse-Bott torus that contains no other asymptotic orbits,
thus the orbits of each $u \in \mM^\cc_0$ are all distinct and simply covered.
It follows that embedded curves form an open subset of $\mM^\cc_0$, which is
also non-empty since it contains~$u_0$.  By positivity of intersections, it
is also closed, so the claim follows from the assumption that
$\mM^\cc_0$ is connected.
\end{proof}

Topologically, $\mM^\cc$ is a closed subspace of $\mM^*$.
Recall that $\mM^*$ can locally be
identified (up to symmetries) with the zero set of the nonlinear 
Cauchy-Riemann operator $\dbar_J$, regarded as a smooth section of a 
certain Banach space bundle.  The same is true for $\mM^\cc$, but with
Banach spaces of maps whose behavior at the ends satisfies exponential 
weighting constraints determined by the numbers $c_z$.  We refer to
\cites{Wendl:automatic,Wendl:BP1} for details on the general analytical setup,
and \cites{HWZ:props3,Wendl:compactnessRinvt,Wendl:BP1} for the exponential
weights.  A given curve $u \in \mM^\cc$ is called 
{Fredholm regular} if the linearization of
$\dbar_J$ at $u$ is surjective.  In general, this linearization
is a Fredholm operator, whose index (with correction terms for the
dimensions of Teichm\"uller space and the automorphism group) defines
the {``virtual dimension''} of the moduli space near~$u$.  {We'll denote
this virtual dimension by $\ind(u ; \cc)$, and call it the (constrained)
index of~$u$.}  {If $u$ is Fredholm regular, then the implicit function
theorem implies that} $\mM^\cc$ near $u$ is a smooth manifold,
{whose dimension is given by the index.}

\begin{Thm}\label{thm:IFT}
Every $u \in \mM^\cc_0$ is Fredholm regular and has $\ind(u ; \cc) = 2$.  
Moreover, a neighborhood of $u$ in $\mM^\cc_0$ forms a smooth
$2$--parameter family $\{ u^\tau \}_{\tau \in \D}$, with $u^0 = u$, such that:
\begin{enumerate}
\item
The images $u^\tau(\dot{\Sigma})$
foliate a neighborhood of $u(\dot{\Sigma})$ in~$W$.
\item
For any puncture $z \in \Gamma_U$, the set of all curves 
$\{ u^\tau \}_{\tau \in \D}$ that approach the same orbit as $u$ at $z$
is a smooth $1$--dimensional submanifold.
\end{enumerate}
\end{Thm}
\begin{proof}
We first verify the claim that $\ind(u ; \cc) = 2$.  For the case where $u$ is
a closed embedded sphere with self-intersection zero, this follows immediately
from the adjunction formula: $0 = u \bullet u = c_1(u^*TW^\infty) - 2$,
thus $c_1(u^*TW^\infty) = 2$ and $\ind(u) = -2 + 2 c_1(u^*TW^\infty) = 2$.

In the case where $u_0$ arises from property~$(\star)$, it 
suffices to prove that $\ind(u_0 ; \cc) = 2$
with $u_0$ regarded as a $J_+$--holomorphic curve in $\R \times M$.
Recall from \cite{Wendl:compactnessRinvt} that one can associate with
$u_0$ an integer $c_N(u_0 ; \cc)$, called the (constrained) normal Chern
number, which satisfies
\begin{equation}\label{eqn:cnIndex}
2 c_N(u_0 ; \cc) = \ind(u_0 ; \cc) - 2 + 2g + \#\Gamma_0(\cc),
\end{equation}
where $g$ is the genus of $\dot{\Sigma}$ (in this case zero) and
$\Gamma_0(\cc)$ is the subset of punctures $z \in \Gamma$ at which
$p(\gamma_z - c_z) = 0$.  It also satisfies
\begin{equation}\label{eqn:windpi}
c_N(u_0 ; \cc) = \windpi(u_0) + \sum_{z \in \Gamma} \left[
\alpha^\Phi_-(\gamma_z - c_z) - \wind^\Phi(e_z) \right].
\end{equation}
By Lemma~\ref{lemma:constraints} and the fact that $u_M : \Sigmadot \to M$
is an embedding, the right hand side of \eqref{eqn:windpi} vanishes,
implying $c_N(u_0 ; \cc) = 0$.  We claim also that $\#\Gamma_0(\cc) = 0$,
i.e.~all punctures satisfy $p(\gamma_z - c_z) = 1$; for $z \in \Gamma_C$
this already follows from Lemma~\ref{lemma:constraints}.  For unconstrained
punctures $z \in \Gamma_U$, Lemma~\ref{lemma:MBintersect} implies that
$e_z$ has the same winding number as a nontrivial section in
$\ker\mathbf{A}_z$: these also span the two eigenspaces of 
$\mathbf{A}_z - c_z = \mathbf{A}_z - \epsilon$ with negative eigenvalues
closest to zero.  It follows that every positive eigenvalue of
$\mathbf{A}_z - \epsilon$ has strictly larger winding, thus 
$p(\gamma_z - \epsilon) = 1$ as claimed.  Now \eqref{eqn:cnIndex} implies
$\ind(u_0 ; \cc) = 2$.

The remainder of the proof consists of minor generalizations of well 
established results from \cites{HWZ:props3,Wendl:thesis}, so we shall merely
sketch the main ideas.  Since $u \in \mM^\cc_0$ is embedded, 
the regularity question can be reduced to the study
of the \emph{normal} Cauchy-Riemann operator $\mathbf{D}_u^N$ as in 
\cites{HoferLizanSikorav,HWZ:props3,Wendl:automatic}.  The domain of
$\mathbf{D}_u^N$ is an exponentially weighted Banach space of sections of
the normal bundle $N_u \to \dot{\Sigma}$, and the sections in 
$\ker\mathbf{D}_u^N$ have only positive zeroes, whose algebraic count is
bounded in general by $c_N(u ; \cc)$, cf.~\cite{Wendl:automatic}.
In our case $c_N(u ; \cc) = c_N(u_0 ; \cc) = 0$, thus
every section in $\ker\mathbf{D}_u^N$ is zero free;
a simple linear independence argument then shows that 
$\dim\ker\mathbf{D}_u^N \le 2 = \ind \mathbf{D}_u^N$, hence 
$\mathbf{D}_u^N$ is surjective.  This shows that $\mM^\cc_0$ is a smooth
$2$--manifold near $u$, and $T_u \mM^\cc_0$ is identified with a space 
of smooth nowhere vanishing sections $\ker\mathbf{D}_u^N \subset \Gamma(N_u)$, 
implying the claim that the curves near $u$ foliate a neighborhood.

Finally we note that for each $z \in \Gamma_U$, one can apply an additional
constraint to study subspaces of curves in $\mM^\cc_0$ that fix the
position of the asymptotic orbit.
In the linearization this amounts to replacing $c_z = \epsilon$ by
$c_z = -\epsilon$; this idea is explained in detail in 
\cites{Wendl:thesis,Wendl:BP1}.  The problem with the additional constraint
then has index~$1$ and is again regular by an argument using the formal adjoint
of $\mathbf{D}_u^N$, as in \cite{Wendl:automatic}.
\end{proof}

Note that in the above proof, Fredholm regularity does not require any
genericity assumptions, rather it comes for free due to
``automatic'' transversality (cf.~\cite{Wendl:automatic}).  As a consequence,
{$u_0$ can be deformed with sufficiently small perturbations of $J$ and 
$\lambda$ so that Theorem~\ref{thm:IFT} still applies.  After such a 
perturbation (using Lemma~\ref{lemma:nondegeneracy}), we can therefore 
assume the following from now on:}
\begin{enumerate}
\item {All orbits of period less than some large constant $C > 0$ are
Morse-Bott}.
\item $J$ is generic outside of $[T,\infty) \times M$, so that in particular
every curve $u \in \mM^\cc$ that isn't wholly contained
in $[T,\infty) \times M$ has $\ind(u ; \cc) \ge 0$.
\end{enumerate}
The exact details of our generic perturbation of $J$ are somewhat
delicate and specific to the application we have in mind; this will be
explained in Lemma \ref{lemma:genericity} in \S\ref{sec:proofs}.  
{Note that
the purpose of this assumption has nothing to do with the curves in
$\mM^\cc_0$, which are already regular---rather we will see below that 
genericity is needed to gain control over the degenerations that can occur
in the natural compactification of~$\mM^\cc_0$.}

{Due to the Morse-Bott assumption,}
the compactness theorem of \cite{SFTcompactness} now applies to
any sequence of $J$--holomorphic curves in $W^\infty$ that satisfy a suitable
$C^0$--bound and energy bound: in particular, such a sequence has a 
subsequence that converges to a {nodal holomorphic building},
typically with multiple {levels}.  In our situation, the bottom level
will be a nodal $J$--holomorphic curve in $W^\infty$, and all levels above
this are nodal $J_+$--holomorphic curves in $\R \times M$.

\begin{Thm}\label{thm:compactness}
Suppose $u_k \in \mM^\cc_0$ is a sequence whose images are all contained
in $W_0 \cup E_+$ for some compact subset $W_0 \subset W$.  Then a
subsequence of $u_k$ converges to one of the following:
\begin{enumerate}
\item another smooth curve in $\mM^\cc_0$,
\item a {holomorphic building with empty bottom level and one
nontrivial upper level that consists of}
a smooth, embedded $J_+$--holomorphic curve
in $\R\times M$ satisfying the conditions of property~$(\star)$, or
\item a nodal $J$--holomorphic curve in $W^\infty$
with {exactly two components, both in $\mM^\cc$ and
both embedded with (constrained) index~$0$.}
\end{enumerate}
Moreover the set of index~$0$ curves that can appear as components of
nodal curves in the third case is finite.
\end{Thm}

Before we prove the theorem we state the following important corollary. 
For this, we denote by $S\subset W^\infty$ the set through which 
the finitely many limit curves from part~(3) of Theorem~\ref{thm:compactness} 
pass, {and let $C \subset W^\infty \setminus S$ consist of all points
that are contained in curves from $\mM^\cc_0$.}

\begin{Cor}\label{cor:foliation}
In addition to the assumptions of Theorem \ref{thm:compactness}, assume that the images of all curves in $\mM^\cc_0$ are contained in $W_0\cup E_+$ for some compact subset $W_0\subset W$. Then $C = W^\infty \setminus S$, and thus $W$ is compact.
\end{Cor}

\begin{proof}
We claim that $C$ is a non-empty, open and closed subset of
$W^\infty\setminus S$.  It is clearly non-empty since 
$\mM^\cc_0$ also is, by construction.
Openness is a direct consequence of Theorem~\ref{thm:IFT} part~(1). 
To prove that $C$ is closed, we choose a sequence $(p_n)\subset C$ with 
$p_n\to p^*\in W^\infty\setminus S$. Then by definition,
there exist curves $u_n\in\mM^\cc_0$ with $p_n\in\im(u_n)$.  A subsequence of
$u_n$ converges to a holomorphic building $u^*$, which by Theorem~\ref{thm:compactness}
is either a smooth curve or a nodal curve with one level.
Since $p^*$ is in the image of $u^*$ and $p^*\not\in S$, we conclude
that $u^*\in\mM^\cc_0$ and $p^*\in\im u^* \subset C$. 

Now, since $S$ is a finite union of images of holomorphic curves,
$W^\infty \setminus S$ is connected and it follows from the above claim that
$C = W^\infty \setminus S$.  Since by assumption $C\subset W_0\cup E_+$, we conclude that $W$ is compact.
\end{proof}

In proving Theorem~\ref{thm:compactness}, we will make use of a few concepts
from the intersection theory of punctured holomorphic curves; this theory
is developed in detail in the papers \cites{Siefring:intersection,SiefringWendl}, 
and the last
section of \cite{Wendl:automatic} also contains a summary.
Assume {$v_1, v_2 \in \mM^\cc$}.
Then there is an algebraic intersection number
$$
i(v_1 ; \cc\ |\ v_2 ; \cc) \in \Z
$$
which has the following properties:
\begin{enumerate}
\item $i(v_1 ; \cc \ |\ v_2 ; \cc)$ is unchanged under {continuous}
variations of $v_1$ and $v_2$ in {$\mM^\cc$}.
\item If $v_1$ and $v_2$ are not both covers of the same somewhere injective
curve, then 
$$
i(v_1 ; \cc \ |\ v_2 ; \cc) \ge 0,
$$
and the inequality is strict if they intersect.
\end{enumerate}

Unlike the usual homological intersection theory applied to closed
holomorphic curves, the last statement is \emph{not} an ``if and only if'':
it is possible in general for $v_1$ and $v_2$ to be disjoint even if
$i(v_1 ; \cc \ |\ v_2 ; \cc) > 0$, though this phenomenon is in some sense
non-generic.  The intersection number can also be defined for curves in 
the symplectization $\R \times M$, possibly with both positive and negative 
punctures.  In this case one has invariance under $\R$--translation,
so if $i(v_1 ; \cc \ |\ v_2 ; \cc) = 0$ then the projected images of
$v_1$ and $v_2$ in $M$ never intersect.

\begin{Lemma}
\label{lemma:selfIntersection}
$i(u_0 ; \cc \ |\ u_0 ; \cc) = 0$.
\end{Lemma}
\begin{proof}
Since $u_0$ has only simply covered Reeb orbits and all of them are distinct,
it satisfies the following somewhat simplified version of the adjunction
formula from \cites{Siefring:intersection,SiefringWendl},
\begin{equation}\label{eqn:adjunction}
i(u_0 ; \cc \ |\ u_0 ; \cc) = 2\delta(u_0) + c_N(u_0 ; \cc).
\end{equation}
Here $\delta(u_0)$ is the algebraic count of double points and singularities
of $u_0$ (see \cite{McDuffSalamon:Jhol}), which vanishes since $u_0$ is
embedded.  As we saw in the proof of Theorem~\ref{thm:IFT},
$c_N(u_0 ; \cc)$ also vanishes, so the claim follows.
\end{proof}

\begin{Lemma}\label{lemma:nicelyEmbedded}
If $v \in \mM^\cc_0$ is contained in $[T,\infty) \times M \subset
W^\infty$, then its projection to $M$ is embedded.
\end{Lemma}
\begin{proof}
Write $v = (v_\R,v_M) : \dot{\Sigma} \to [T,\infty) \times M$.
By assumption, $v$ can be deformed continuously to $u_0$ through
$\mM^\cc$, thus $i(v ; \cc \ |\ v ; \cc) = i(u_0 ; \cc \ |\ u_0 ; \cc) = 0$
by the previous lemma, and $c_N(v ; \cc) = c_N(u_0 ; \cc) = 0$.
Now \eqref{eqn:windpi} implies that $\windpi(v) = 0$, thus $v_M$ is immersed,
and the vanishing self-intersection number implies that $v$ has no
intersections with any of its $\R$--translations, so $v_M$ is also injective.
\end{proof}

\begin{proof}[Proof of Theorem~\ref{thm:compactness}]
By \cite{SFTcompactness}, $u_k$ has a subsequence converging to some
holomorphic building, which we'll denote by~$u$.  Our {first} task is to
show that unless $u$ is {a} $2$--level building {with empty bottom
level} {as} described in case~(2),
it can have no nontrivial upper levels.  This is already clear in the case
where $u$ is closed, as convexity prevents $u_k$ from venturing
into the region $[T,\infty) \times M$ at all.  Let us therefore assume
that $u_k$ has punctures and that $u$ has nontrivial upper levels.
If no component in these upper levels has any negative punctures,
then there must be only one nontrivial level, which consists of one or more
connected components $v_1,\ldots,v_N$ attached to each other by nodes.
All of these components have punctures, since the symplectic form in
$\R\times M$ is exact; moreover, the positive ends of each $v_i$ correspond
to some subset of the positive ends of $u_0$, and since these are all simply
covered and distinct, each $v_i$ is somewhere injective and satisfies
the asymptotic constraints defined by~$\cc$.  Now
\eqref{eqn:cnIndex} and \eqref{eqn:windpi} give
$$
0 \le 2\windpi(v_i) \le 2 c_N(v_i ; \cc) = \ind(v_i ; \cc) - 2,
$$
hence $\ind(v_i ; \cc) \ge 2$.  Since $\ind(u_0 ; \cc) = 2$ as well,
we conclude that $u$ can have at most one connected component, with no nodes,
i.e.~it is a smooth $J_+$--holomorphic curve in $\R\times M$ with only
positive punctures.  Up
to $\R$--translation, $u$ can therefore be identified with some smooth curve
in $\mM^\cc_0$ whose image is contained in $[T,\infty) \times M$, and
the projection into
$M$ is embedded due to Lemma~\ref{lemma:nicelyEmbedded}.  It follows that
this curve satisfies the conditions of property~$(\star)$.

{Alternatively, suppose $u$ has nontrivial upper levels and the top
level} contains a $J_+$--holomorphic curve $u_+$ in $\R \times M$
which is not the trivial cylinder over an orbit and has both positive and
negative punctures.  {Repeating the above argument about behavior at
the positive ends, $u_+$ is somewhere injective.}
Applying Stokes' theorem to $\int u_+^*d\lambda \ge 0$,
the negative asymptotic 
orbits of $u_+$ have total period bounded by the total period of the
positive orbits, implying that all of the negative orbits
belong to the same Morse-Bott manifolds as the orbits of $u_0$.
We claim that {after some $\R$--translation, $u_+$ intersects $u_0$.}
This will imply a contradiction almost immediately, as positivity
of intersections then gives an intersection of $u_k$ with some
$\R$--translation of $u_0$ for sufficiently large $k$,
contradicting Lemma~\ref{lemma:selfIntersection} since
$i(u_k ; \cc \ |\ u_0 ; \cc) = i(u_0 ; \cc \ |\ u_0 ; \cc) = 0$.

To prove the claim, it suffices to show that the projected
images of $u_+$ and $u_0$ in $M$ intersect each other.
Suppose $\gamma$ is an
asymptotic orbit of $u_0$ that lies in the same Morse-Bott submanifold
$N \subset M$
as one of the negative asymptotic orbits $\gamma'$ of $u_+$.
Denote the corresponding asymptotic eigenfunctions by $e$ and $e'$
respectively.  We consider the following cases:

\textbf{Case~1: $N$ is a circle.}
Then $\gamma$ is nondegenerate and $\gamma'$ is the $k$--fold cover of
$\gamma$ for some $k \in \N$.  Choose a trivialization $\Phi$ along
$\gamma$ so that $\wind^\Phi(e) = 0$.  By Lemma~\ref{lemma:constraints},
$\mathbf{A}_\gamma$ has two eigenvalues (counting multiplicity) $\mu < 0$
with $\wind^\Phi(\mu) = 0$.  Then the $k$--fold covers of their eigenfunctions 
are eigenfunctions of $\mathbf{A}_{\gamma'}$ with negative eigenvalues
and zero winding, implying that every positive eigenvalue of
$\mathbf{A}_{\gamma'}$ has strictly positive winding.  Thus
$\wind^\Phi(e') \ge \alpha^\Phi_+(\gamma') > 0$, forcing the projections
of $u_0$ and $u_+$ in $M$ to intersect each other near~$N$.

\textbf{Case~2: $N$ is a torus disjoint from $u_M$.}
Now $\gamma'$ can be deformed through a $1$--parameter family of orbits to
a $k$--fold cover of $\gamma$ for some $k \in \N$.  Choose a trivialization
$\Phi$ along every simply covered orbit in $N$ so that sections in the
$0$--eigenspaces have zero winding.
By Lemma~\ref{lemma:MBintersect}, $\mathbf{A}_\gamma$ has an eigenvalue
$\mu < 0$ such that $\wind^\Phi(e) = \wind^\Phi(\mu) = 0$,
and taking $k$--fold covers of eigenfunctions, we similarly find
eigenfunctions of $\mathbf{A}_{\gamma'}$ that have zero winding and
eigenvalues $k\mu < 0$ and $0$.  This implies that
$\wind^\Phi(e') \ge \alpha^\Phi_+(\gamma') > 0$, which forces the projection 
of $u_+$ in $M$ to intersect $N$, i.e.~$u_+$ intersects a trivial cylinder
$\R \times \gamma_1$ for some orbit $\gamma_1 \subset N$.
Then by the homotopy invariance of the intersection number, $u_+$ also
intersects $\R \times \gamma$.  This intersection is transverse unless
it occurs at a point where $\pi_\lambda \circ Tu_+ = 0$, but the similarity
principle implies that there are finitely many such points 
(see \cite{HWZ:props2}).  Thus if necessary we can use Theorem~\ref{thm:IFT}
to perturb $u_0$ and thus move $\gamma$ to a nearby orbit, so that the
intersection of $\R\times \gamma$ with $u_+$ is transverse.  This implies
a transverse intersection of the projected image of $u_+$ in $M$ with $\gamma$,
and therefore an intersection of the projections of $u_+$ and $u_0$ nearby.

\textbf{Case~3: $N$ is a Morse-Bott manifold intersecting $u_M$.}
The argument is similar to case~2, only now we use the intersection of
$u_M$ with $N$ to show that $u_M$ intersects $\gamma'$ and thus also
the projected image of $u_+$ near~$\gamma'$.

We've shown now that $u$ cannot have any nontrivial upper level
{except in case~(2)}, so
it must therefore be a $1$--level building in $W^\infty$, i.e.~a nodal
$J$--holomorphic curve.  The deduction of case~(3) now proceeds almost exactly
as in the proof of \cite{Wendl:fillable}*{Theorem~7}.  To summarize,
the connected components of $u$ are all either punctured curves with
positive ends at distinct simply covered orbits (and thus somewhere injective),
or closed curves (which must be nonconstant by an
index argument).  The latter could in general be multiple covers, but if
$v$ is a $k$--fold branched cover of some closed somewhere injective curve
$v_0$, then we find $\ind(v) = k \cdot \ind(v_0) + 2 (k - 1)$.
Due to our genericity assumption, all somewhere injective curves have
index at least~$0$, so we find that the total index of $u$ becomes more
than~$2$ unless there is at most one node connecting two components,
and in this case both components must be somewhere injective.
The adjunction
formula \eqref{eqn:adjunction} can now be used to
show that these two components, $v_1$ and $v_2$, are both embedded,
satisfy $i(v_i ; \cc\ |\ v_i ; \cc) = -1$, $i(v_i ; \cc\ |\ u_0 ; \cc) = 0$ 
and $i(v_1 ; \cc \ |\ v_2 ; \cc) = 1$;
moreover, they are both Fredholm regular and have (constrained) index~$0$.

There's one minor point to address which was irrelevant in
\cite{Wendl:fillable}: if there are no punctures, we haven't ruled out the
possibility that $u$ is a smooth multiple cover,
i.e.~$u = v \circ \varphi$ for some closed somewhere injective sphere
$v$ and holomorphic branched cover $\varphi : S^2 \to S^2$.  Since
$c_1(u^*TW^\infty) = 2$, this is allowed numerically only if
$c_1(v^*TW^\infty) = 1$ and $\varphi$ has degree~$2$.
But then we get a simple contradiction using the
adjunction formula: since $u \bullet u = 0$, the same holds for $v$, thus
$$
0 = v \bullet v = 2\delta(v) + c_1(v^*TW^\infty) - 2 = 2\delta(v) - 1
$$
where $\delta(v)$ is the algebraic count of double points and singularities.
The right hand side is odd; in particular it can never be zero.

It remains to show that the set of all index~$0$ curves arising from nodal
degenerations of $u_k$ is finite.
Indeed, suppose $v_k$ is a sequence of finite energy
$J$--holomorphic curves in $W^\infty$ with uniform energy and $C^0$--bounds
such that
\begin{enumerate}
\item
The punctures of $v_k$ are identified with a subset of $\Gamma$ and
satisfy the asymptotic constraints of Definition~\ref{def:constrained}.
\item
$i(v_k ; \cc \ |\ u_0 ; \cc) = 0$.
\item
$\ind(v_k ; \cc) = 0$.
\end{enumerate}
Then we claim that $v_k$ has a convergent subsequence.  The argument is
familiar: we rule out nontrivial upper levels exactly as before by
showing that any nontrivial component $v_+$ in such a level must intersect
$u_0$.  Thus the only remaining possible non-smooth
limit is a nodal curve in $W^\infty$, but the same index argument 
now implies that there is at
most one component, thus no nodes, and the limit is somewhere injective.  
It follows that this set of curves is
a compact smooth $0$--dimensional manifold, i.e.~a finite set.
\end{proof}

\section{Proofs of the main results}
\label{sec:proofs}

\subsection{Proofs of Theorems \ref{thm:rationalRuled} and \ref{thm:separating}}

We consider a closed and connected symplectic $4$--manifold $(W,\om)$ which contains a closed contact type hypersurface $M$
such that $W \setminus M$ is connected.  Under the assumptions of
Theorem~\ref{thm:rationalRuled} or~\ref{thm:separating}, we will construct
from this a noncompact symplectic manifold with convex boundary to which
Corollary~\ref{cor:foliation} applies, giving a contradiction.
The general idea of the construction is outlined in 
Figure~\ref{fig:infinitechain}.

To start with, we compactify $W\setminus M$ by adding to each end a copy of $M$,
obtaining a compact and connected symplectic manifold $(W_1,\om)$ with one convex 
boundary component $M^+$ and an identical concave boundary component $M^-$. 
Inductively, we define the compact symplectic 
manifold $W_n$ by $W_n:=W_{n-1}\cup_{M^-=M^+}W_1$, denoting the symplectic 
form on $W_n$ again by $\om$. Note that $W_{n-1}$ is a compact symplectic submanifold of $W_n$ in a natural way. Thus the set
\beq
(\mathcal{W},\om):=\bigcup_{n\geq1}(W_n,\om)
\eeq
is a noncompact symplectic manifold with convex boundary $M$ corresponding to the convex boundary of $W_1$.

Assume that $W$ contains a symplectically embedded sphere $S \subset W$
with $S \bullet S = N \ge 0$. {Since $\omega$ is exact on $M$, Stokes'
theorem implies that $S$ cannot be contained entirely in~$M$.
We can thus blow up $W$} at $N$ distinct points
in $S$ {that are} not in $M$, {modifying} both $W$ and $S$ so that 
$S \bullet S = 0$ without loss of generality.  Now we claim that $S$
can be ``lifted'' to a symplectic sphere $\widetilde{S}$ in 
$(\W,\om)$ with $\widetilde{S} \bullet \widetilde{S} = 0$.  To see this, construct
a symplectic infinite cover $(\widetilde{W},\widetilde{\omega})$ of $(W,\omega)$ by 
gluing together a sequence of copies $\{ (W_1^j,\omega) \}_{j \in \Z}$ 
of $(W_1,\omega)$, with the concave boundary of $W_1^j$ attached to the
convex boundary of $W_1^{j+1}$ for each $j \in \Z$.  Since the sphere is
simply connected, $S$ has a lift $\widetilde{S} \subset \widetilde{W}$,
and moreover, $(\widetilde{W},\widetilde{\omega})$ naturally contains
$(\W,\om)$, which we may assume contains $\widetilde{S}$ without loss of 
generality.  

Similarly, if $M$ with its induced contact structure satisfies weak property 
$(\star)$, then after
attaching a symplectic cobordism to the convex boundary of $(\W,\om)$,
we may assume without loss of generality that {either $(\W,\om)$ contains
a symplectic sphere of zero self-intersection (after blowing up) or
property~$(\star)$ holds for $\p\W$.}

In either case, $(\W,\om)$ now satisfies Assumption~\ref{assumptions}.
As explained in \S\ref{sec:compactness}, we can then attach to $\p\W$
a cylindrical end $E^+$ that contains $([T,\infty) \times M, d(e^t\lambda))$
for sufficiently large $T \in \R$ and a suitable contact form $\lambda$,
obtaining an enlarged symplectic manifold $(\W^\infty,\om)$,
with an $\om$--compatible almost complex structure $J_0$ that is 
admissible and $\R$--invariant on $[T,\infty) \times M$, and a 
non-empty moduli space $\mM^\cc_0 \subset \mM^\cc$ of $J_0$--holomorphic curves in $\W^\infty$.
{Moreover for some $n_0 \in \N$, we can assume that $J_0$ belongs
to the following set.}

\begin{Def}\label{def:periodic_acs}
Let $\J^\mathrm{per}$ be the space of compatible almost complex structures 
on $(\W^\infty,\om)$ which match $J_0$
on $([T,\infty) \times M, d(e^t\lambda))$ and whose restrictions
to $W\cong W_{n+1}\setminus W_n\subset \W^\infty$ are independent of $n$ for ${n\geq n_0(J_0)}$. 
Such a $J$ will be called periodic.
\end{Def}

\begin{Lemma}\label{lemma:genericity}
For a generic $J\in\J^\mathrm{per}$, all $J$--holomorphic curves in 
$\mM^\cc$ are Fredholm regular.
\end{Lemma}
\begin{proof}
Recall that the $J$--holomorphic curves in $\mM^\cc$ are somewhere 
injective, see \S\ref{sec:compactness}.  The proof of transversality is a
small variation on the standard technique, as in \cite{McDuffSalamon:Jhol}: 
the key is to show
that the universal moduli space $\{(u,J)\mid u\text{ is $J$--holomorphic}\}$
is a smooth Banach manifold for periodic $J$ and $u$
satisfying the relevant conditions.  This will use the fact that a perturbation 
of $J$ can be localized at an injective point of $u$ without interfering at 
other points in the image of $u$.  Then regular values of the projection 
$(u,J)\mapsto J$ are generic by the Sard-Smale theorem, and for these,
all $J$--curves are Fredholm regular.

Assume $J \in \J^\mathrm{per}$ and $u \in \mM^\cc$ is not fully contained in
$[T,\infty) \times M$.  If $u$ also intersects $W_{n_0} \cup E^+$,
then it suffices to perturb $J$ only in this region and thus preserve
periodicity of $J$.  Thus it remains only to show that $\J^\mathrm{per}$ permits
sufficient perturbations of $J$ when the image of $u$ is contained in
$\W \setminus W_{n_0}$, in which case $u$ must be a somewhere injective
closed curve.
Since $J$ is required to be periodic, the only danger not present in the
standard case is that $u$ may have \emph{periodic points}, in the following 
sense.  Recall that $\W^\infty$ contains infinitely many identical copies 
of a certain manifold $V$, in the form 
$\widehat{W}_n := W_{n+1}\setminus W_n$. Thus each point $x\in V$ appears 
infinitely often in $\W^\infty$, and we call these different points 
translates of $x$.  Then $z\in\dot{\Sigma}$ is a periodic point of $u$ if 
a translate of $u(z)$ is contained in the image $\im(u)$ of $u$. 
In this case a periodic perturbation {of $J$} cannot be localized in the image of $u$.

We claim that for any somewhere injective closed holomorphic curve in
$\W \setminus W_{n_0}$, the set of injective points which are not periodic 
is open and dense.  To see this, we can consider
the covering space $\pi : \widetilde{W} \to W$ 
which was constructed above Definition~\ref{def:periodic_acs}.
Since $J$ is periodic, the projection $\pi \circ u$ is a
holomorphic curve in $W$. It will suffice to show that also $\pi \circ u$
is somewhere injective, since then the set of injective points of $\pi\circ u$ is open and dense,
{and injective points} of $\pi\circ u$ give rise to non-periodic injective points of $u$.  Denote by 
$\tau : \widetilde{W} \to \widetilde{W}$ the deck transformation
that maps $\widetilde{W}_n$ to $\widetilde{W}_{n+1}$.  Then if $\pi \circ u$
is multiply covered, the fact that $u$ is somewhere injective implies
(using unique continuation)
that $u$ and $\tau^k \circ u$ are equivalent curves for some integer
$k \ne 0$.  But then $u$ is also equivalent to $\tau^{nk} \circ u$ for any
$n \in \Z$, implying that the image of $u$ in $\widetilde{W}$ is unbounded.
Since $u$ was assumed to be closed, this is a contradiction and shows that
$\pi \circ u$ is indeed somewhere injective.

With this, the usual proof that the universal 
moduli space is a smooth Banach manifold goes through unchanged. 
\end{proof}

For the remainder of this section we assume that the almost complex 
structure $J$ {(formerly called $J_0$)} is periodic {and generic}.

\begin{Prop}\label{prop:uniform_compactness}
There exists $N_0\in\N$ such that for all $u\in\mM^*_0$ we have
\beq
\im(u)\subset W_{N_0} \cup E^+.
\eeq
\end{Prop}

\begin{proof}
We denote the convex boundary of $W_n\subset\W$ by $M^+$ and the  concave boundary by $M^-_n$. Recall that $M^+$ is the same for all $W_n$. Then we claim that there exists a positive constant $c_0>0$  such that all $u\in\mM^*_0$ with $\im(u)\cap M^+$ and $\im(u) \cap M^-_n$ both nonempty have energy
\beq\label{eqn:energy_grows}
E(u)\geq c_0n\;.
\eeq
This follows from the monotonicity lemma (see Lemma~\ref{lemma:monotonicity} below) and the fact that the almost complex structure is periodic. Indeed, we fix a copy of $W_1$ in $W_n$ and denote 
for the moment its convex  and concave boundary by $\p W^+$ and $\p W^-$ respectively. We claim that there exists $\tilde{c}>0$ such that any holomorphic curve $v$ with $v^{-1}(\p W^+) \neq\emptyset$ and $v^{-1}(\p W^-)\neq\emptyset$ has at least energy $E(v)\geq\tilde{c}$. To see this we observe that each such $v$ has to pass through a point in $W$ with distance $\epsilon_0>0$ to the boundary $\p W^+\cup\p W^-$ of $W$. Thus we conclude from Lemma~\ref{lemma:monotonicity} that $E(v)\geq C\epsilon_0^2$ for each $v$, where $C$ and $\epsilon_0$ only depend on the almost complex structure $J$. Since $J$ is periodic, and a map $u\in\mM^*_0$ with $\im(u)\cap M^+\neq\emptyset$ and $\im(u)\cap M^-_n \neq\emptyset$ passes through the boundaries of $n$ copies of $W_1$, equation \eqref{eqn:energy_grows} follows.  Using the uniform energy bound for $u \in \mM^*_0$, this implies the proposition in the case where $u_0$ has punctures, as every $u \in \mM^*_0$ is then either confined to $E^+$ or passes through~$M^+$.

A small modification is required for the case without punctures: here $u_0 \in \mM^*_0$ is a sphere, and we can choose its lift from $W$ to $\W^\infty$ so that without loss of generality, the image of $u_0$ intersects~$W_1$ (i.e.~the first copy).  Then we claim that \emph{every} $u \in \mM^*_0$ intersects~$W_1$.  Otherwise, the fact that $\mM^*_0$ is connected implies the existence of some holomorphic sphere in $\mM^*_0$ that touches $M^-_1$ tangentially from inside $W_2 \setminus W_1$, and this is impossible by convexity.  We conclude that every $u \in \mM^*_0$ which escapes from $W_1 \cup E^+$ must also pass through $M^-_1$, so the above argument goes through by using $M^-_1$ in place of~$M^+$.
\end{proof}

For the sake of completeness, we include here the monotonicity lemma,
see \cite{Hummel} for a proof.
\begin{Lemma}
\label{lemma:monotonicity}
For any compact almost complex manifold $(W,J)$ with Hermitian metric
$g$, there are constants $\epsilon_0$ and $C > 0$ such that the following
holds.  Assume $(S,j)$ is a compact Riemann surface, possibly with boundary,
and $u : S \to W$ is a pseudoholomorphic curve.  Then for every
$z \in \text{Int}(S)$ and $r \in (0,\epsilon_0)$ such that
$u(\p S) \cap B_r(u(z)) = \emptyset$, the inequality
$$
\operatorname{Area}\left(u(S) \cap B_r(u(z))\right) \ge C r^2
$$
holds.
\end{Lemma}

Since $W_{N_0}$ is compact, Proposition \ref{prop:uniform_compactness} 
allows us to apply Corollary \ref{cor:foliation}.  But this implies that
$\W$ is compact, and is thus a contradiction, concluding the proof of
Theorems~\ref{thm:rationalRuled} and~\ref{thm:separating}.

\begin{figure}[htb]
\psfrag{W}{$(W,\omega)$}
\psfrag{M}{$(M,\xi)$}
\psfrag{Vn}{$(W_n,\omega)$}
\psfrag{VncupEplus}{$W_n\cup E^+$}
\psfrag{Y}{$Y$}
\psfrag{V1}{$(W_1,\omega)$}
\includegraphics[scale=.48]{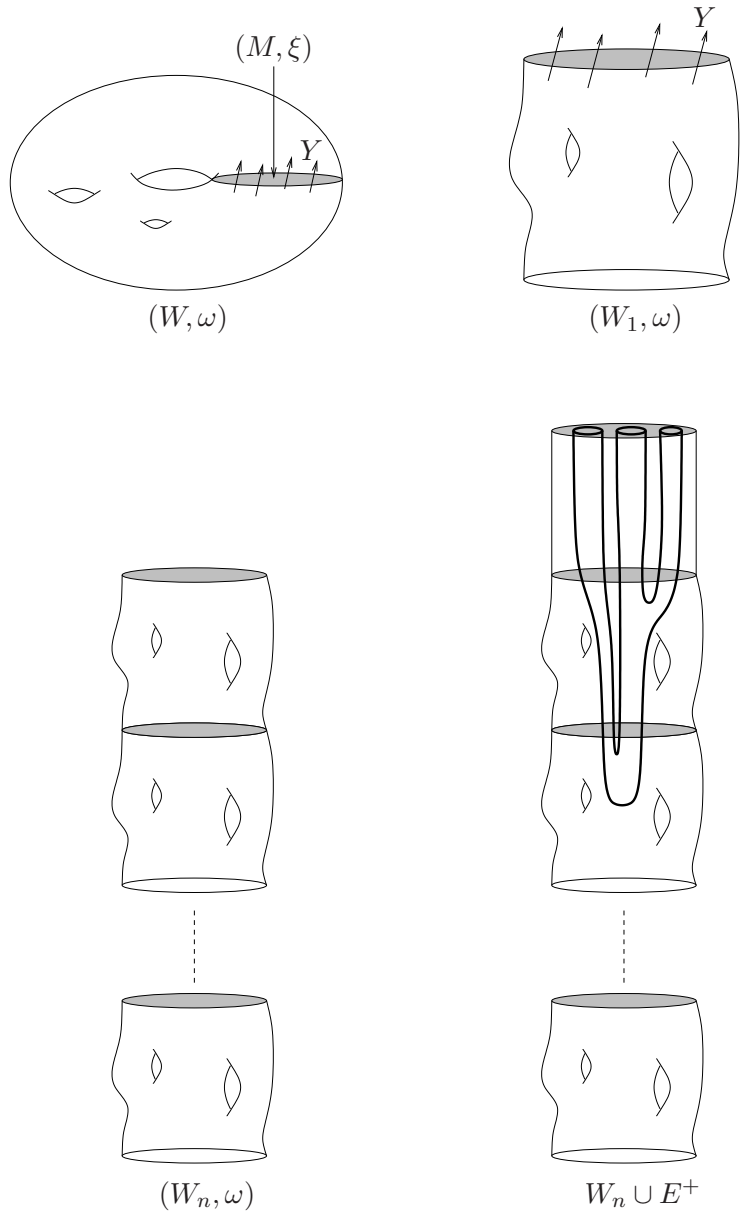}
\caption{\label{fig:infinitechain} The compact symplectic manifold 
$(W,\omega)$ contains the 
non-separating contact hypersurface $(M,\xi)$.  $W\setminus M$ is 
compactified to produce $(W_1,\omega)$, which has two boundary components 
contactomorphic to $M$, one convex and one concave.  Successively 
attaching $n$ copies of $W_1$ to itself produces $(W_n,\omega)$.
Then property~$(\star)$ gives rise to a moduli space of finite energy curves
which, due to the monotonicity lemma, cannot escape from
$W_n \cup E^+$ if $n$ is sufficiently large.}
\end{figure}

\subsection{Proof of Theorem \ref{thm:multiple}}
\label{subsec:moreProofs}

Theorem~\ref{thm:multiple} follows immediately from
Theorem~\ref{thm:separating} and Example~\ref{ex:EtnyresExample}, since
a symplectic semifilling with disconnected boundary can always be turned
into a closed symplectic $4$--manifold containing non-separating contact
hypersurfaces.  One can nonetheless give a slightly easier proof as follows.

Assume that the boundary $\partial W$ is disconnected and contains a component $M$ satisfying property~$(\star)$.  Thus $W$ satisfies Assumption \ref{assumptions}, and after attaching cylindrical ends, we obtain a moduli space $\mM^\cc_0$ of $J$--holomorphic curves that fill the enlarged manifold $W^\infty$. Moreover, all $J$--holomorphic curves have positive punctures going to the end corresponding to~$M$.  Since they fill $W^\infty$, some of these curves must therefore touch $\p W \setminus M$ tangentially, which is impossible if $\p W$ is convex.

\subsection{Proof of Theorem \ref{thm:separating_general}}
\label{subsec:evenMoreProofs}

Let $(W,\om)$ be a compact connected 4-manifold with convex boundary 
$(M,\xi)$ satisfying the weak $(\star)$ property.  After attaching a symplectic
cobordism to $\p W$, we may without loss of generality remove the word 
``weak''.  Now assume that $H \subset W\setminus M$ is a non-separating 
contact hypersurface.  Thus we can cut $W$ open along $H$ and compactify to
obtain a connected symplectic cobordism $W_1$ with two convex boundary
components $H^+$ and $M$, and one concave boundary component $H^-$.

Now we can repeat the construction in the proof of Theorems~\ref{thm:rationalRuled} and~\ref{thm:separating}, namely we glue infinitely many copies of $W_1$ along $H$, obtaining a noncompact symplectic manifold $\W$ with one convex boundary component $H$ and infinitely many convex boundary components which are copies of $M$.  From here, we proceed exactly as in the previous proofs, using the moduli space of holomorphic curves arising from property~$(\star)$ on the first copy of~$M$.  The only new feature is that $\p\W$ is not compact, but since it consists of copies of the same compact and convex components, the results of \S\ref{sec:compactness} still hold, as convexity prevents the holomorphic curves in $\mM^\cc_0$ from ever approaching the other copies of $M$.  In particular, Corollary~\ref{cor:foliation} applies and again yields a contradiction.

\appendix

\section{Relative nondegeneracy of contact forms}
\label{sec:nondegeneracy}

{Our main argument uses holomorphic curves asymptotic to Morse-Bott 
families of periodic orbits. We prefer not to assume from the start that
the contact form is \emph{globally} Morse-Bott. Thus, we need a perturbation
result that preserves a given Morse-Bott submanifold and makes $\lambda$ nondegenerate everywhere else.}  
For this, it suffices to show that one can perturb $\lambda$ in some precompact subset to make all
orbits that pass through that subset nondegenerate.

\begin{Thm}
\label{thm:nondegeneracy}
Suppose $M$ is a $(2n-1)$--dimensional manifold with a smooth contact form
$\lambda$, and $\U \subset M$ is an open subset with compact closure.
Then there exists a Baire subset
$$
\Lambda_\reg(\U) \subset \{ f \in C^\infty(M) \ |\ 
f > 0 \text{ and } f|_{M\setminus \U} \equiv 1 \}
$$
such that for each $f \in \Lambda_\reg(\U)$, every periodic orbit of
$X_{f\lambda}$ passing through $\U$ is nondegenerate.
\end{Thm}
\begin{proof}
We give a proof in two steps, first showing that a generic choice of the
{function $f$ makes all simply covered orbits of $X_{f\lambda}$ passing through $\U$ 
nondegenerate. Then we extend this to multiple covers by a further
perturbation.}

{The first step is an adaptation of the standard Sard-Smale argument.}  Let $\xi = \ker\lambda$, 
and for some large $k \in \N$, define the Banach space
$$
C^k_\U(M) = \left\{ f \in C^k(M,\R)\ \big|\ f|_{M \setminus \U} \equiv 0 \right\}
$$
and Banach manifold
$$
\Lambda^k(\U) = \{ f \in C^k(M,\R)\ |\ f > 0 \text{ and } f - 1 \in
C^k_\U(M) \},
$$
whose tangent space at any $f \in \Lambda^k(\U)$
can be identified with $C^k_\U(M)$.  We will consider the nonlinear operator
$$
\sigma(x,T,f) := \dot{x} - T X_{f\lambda}(x)
$$
as a section of a Banach space bundle over $H^1(S^1,M) \times (0,\infty)
\times \Lambda^k(\U)$ whose fiber at $(x,T,f)$ is $L^2(x^*TM)$.  Since
$X_{f\lambda}$ depends on the first derivative of $f$, it is of class
$C^{k-1}$ and the section $\sigma$ is therefore of class~$C^{k-2}$.
Choosing any symmetric connection $\nabla$ on~$M$,
the linearization of $\sigma$ at $(x,T,f) \in \sigma^{-1}(0)$
with respect to the first variable defines the operator
\begin{equation}
\label{eqn:Dx}
D_x : H^1(x^*TM) \to L^2(x^*TM) : \hat{x} \mapsto \nabla_t \hat{x} -
T \nabla_{\hat{x}} X_{f\lambda}.
\end{equation}
Since $\dot{x} = T X_{f\lambda}(x)$, we can identify the normal bundle of
$x$ with $x^*\xi$ and thus define a splitting $x^*TM = T S^1 \oplus x^*\xi$.
A short calculation then allows us to rewrite $D_x$ with respect to the
splitting in the block form
\begin{equation}
\label{eqn:block}
D_x = \begin{pmatrix}
\p_t & 0  \\
0     & D_x^N
\end{pmatrix},
\end{equation}
where $D_x^N : H^1(x^*\xi) \to L^2(x^*\xi)$ is defined again by
\eqref{eqn:Dx}, and is a Fredholm operator of index~$0$.  The orbit $x$ is
nondegenerate if and only if $D_x^N$ is an isomorphism.

The total linearization of $\sigma$ at $(x,T,f) \in \sigma^{-1}(0)$ is now
$$
D\sigma(x,T,f)(\hat{x},\hat{T},\hat{f}) = D_x\hat{x} - 
\hat{T} X_{f\lambda}(x) - T \widehat{X}(x),
$$
where we define the vector field $\widehat{X} := 
\p_\tau X_{(f + \tau\hat{f})\lambda}|_{\tau=0}$.
It follows from the definition of the Reeb vector field that $\widehat{X}$
takes the form $-\hat{f} X_{f\lambda} + V_{\hat{f}}$ where 
$V_{\hat{f}} \in \Gamma(\xi)$ is
uniquely determined by the condition
\begin{equation}
\label{eqn:V}
\left. d(f\lambda)(V_{\hat{f}},\cdot)\right|_{\xi} = 
\left. d\hat{f}\right|_{\xi}.
\end{equation}

We define the universal moduli space of parametrized Reeb orbits as
$\mM := \sigma^{-1}(0)$, and let $\mM^* \subset \mM$ denote the open subset 
consisting of triples $(x,T,f)$ for which $x$ is simply covered and
$x(S^1) \cap \U \ne \emptyset$.  Similarly, denote
$$
\mM^*(f) = \{ (x,T) \ |\ (x,T,f) \in \mM^* \}.
$$
We claim
that $D\sigma(x,T,f)$ is surjective whenever $(x,T,f) \in \mM^*$, hence
$\mM^*$ is a $C^{k-2}$--smooth Banach manifold.  To see this, note that
one can always find $\eta \in H^1(T S^1)$ and $\hat{T} \in \R$ so that
$Tx(\p_t \eta) - \hat{T} X_{f\lambda}(x)$ takes any desired value in
$L^2(x^*(\R X_{f\lambda}))$, thus it suffices to show that
the ``normal part'' 
$$
H^1(x^*\xi) \oplus C^k_\U(M) \to L^2(x^*\xi) : (\hat{x},\hat{f}) \mapsto
D_x^N \hat{x} - T V_{\hat{f}}
$$
is surjective.  If it isn't, then there exists a {section
$\eta\neq0\in L^2(x^*\xi)$} such that $\langle D_x^N \hat{x}, \eta \rangle_{L^2} =0$
for all $\hat{x} \in H^1(x^*\xi)$ and 
$\langle V_{\hat{f}} , \eta \rangle_{L^2} = 0$ for all $\hat{f} \in
C^k_\U(M)$ vanishing outside of~$\U$.  The first relation implies that
$\eta$ is in the kernel of the formal adjoint of $D_x^N$, a first order linear
differential operator, hence $\eta$ is smooth and nowhere vanishing.
But then if $x(t_0) \in \U$, then using \eqref{eqn:V}, $\hat{f}$ can be 
chosen near $x(t_0)$ so that the second relation requires $\eta$ to vanish on 
a neighborhood of $t_0$, giving a contradiction.

Now applying the Sard-Smale theorem to the natural projection
$\mM^* \to \Lambda^k(\U) : (x,T,f) \mapsto f$, we find a Baire subset
$\Lambda^k_\reg(\U) \subset \Lambda^k(\U)$ for which every simply covered
Reeb orbit passing through $\U$ is nondegenerate.  

For the second step, denote by $\dist(\ ,\ )$ the distance functions
resulting from any choice of Riemannian metrics on $S^1$ and $M$, and
define for each positive integer $N \in \N$ a subset
$$
\mM_N(f) \subset \mM^*(f)
$$
consisting of Reeb orbits $(x,T)$ that satisfy the following conditions:
\begin{enumerate}
\item $T \le N$.
\item There exists $t \in S^1$ such that
$$
\inf_{t' \in S^1 \setminus\{t\}} \frac{\dist(x(t),x(t'))}{\dist(t,t')} \ge 
\frac{1}{N}.
$$
\item There exists $t \in S^1$ such that $\dist(x(t),M\setminus \U) \ge 1 / N$.
\end{enumerate}
Moreover, let $\Lambda_{\reg,N}(\U) \subset \Lambda^\infty(\U)$ denote the
space of all smooth functions $f \in \Lambda^k(\U)$ for which all
covers of orbits in $\mM_N(f)$ up to multiplicity~$N$ are nondegenerate.
Since nondegeneracy is an open condition and any sequence $(x_k,T_k) \in
\mM_N(f_k)$ with $f_k \to f$ in $C^\infty$ has a convergent subsequence by
the Arzel\`{a}-Ascoli theorem, $\Lambda_{\reg,N}(\U)$ is an open set.
We claim it is also dense.  Indeed, any $f \in \Lambda^\infty(\U)$ has a
perturbation $f_\epsilon \in \Lambda^k(\U)$ for which all the simple orbits
in $\mM_N(f_\epsilon)$ are nondegenerate due to step~1.  In this case
$\mM_N(f_\epsilon)$ is a smooth compact $1$--manifold, i.e.~a finite union
of circles, which are the parametrizations of finitely many distinct
nondegenerate orbits, and the space is stable under small perturbations
of~$f_\epsilon$.  Thus by a further perturbation, we can
make $f_\epsilon$ smooth and arrange that none of the orbits in
$\mM_N(f_\epsilon)$ have a Floquet multiplier that is a $k$th root of unity
for $k \in \{1,\ldots,N\}$.  The latter can be achieved using a normal form
for $f_\epsilon\lambda$ as in \cite{HWZ:props1}*{Lemma~2.3} near each
individual orbit: in particular, we can perturb so that each orbit remains
unchanged but the linearized return map changes arbitrarily within the
space of symplectic linear maps.  This proves that $\Lambda_{\reg,N}(\U)$ is
dense in $\Lambda^\infty(\U)$, and we can now construct
$\Lambda_\reg(\U)$ as a countable intersection of open dense sets:
$$
\Lambda_\reg(\U) = \bigcap_{N \in \N} \Lambda_{\reg,N}(\U).
$$
\end{proof}

\section*{Acknowledgments}
We would like to thank Klaus Mohnke for bringing the
question of non-separating contact hypersurfaces to our attention,
John Etnyre for providing Example~\ref{ex:EtnyresExample},
Klaus Niederkr\"uger, Paolo Ghiggini and Slava Matveyev
for enlightening discussions and Yasha Eliashberg for some helpful
comments on an earlier draft of the paper.
We began work on this article while BB was visiting ETH Z\"urich, and we'd
like to thank ETH for its stimulating working environment.
PA is partially supported by NSF grant DMS-0805085.
CW is partially supported by NSF Postdoctoral Fellowship DMS-0603500.

\end{document}